\documentclass[12pt]{article}
\usepackage[T1]{fontenc}
\usepackage[cp1250]{inputenc}
\usepackage{amssymb}
\usepackage{amsthm}
\usepackage{hyperref}
\usepackage{graphics}
\usepackage[reqno]{amsmath}
\usepackage{indentfirst}
\usepackage{amsmath}
\usepackage{bbm}
\usepackage[left=2cm, right=2cm, top=2cm, bottom=2cm]{geometry}

\newtheorem{prop}{Proposition}[section]
\newtheorem{thm}{Theorem}[section]
\newtheorem{lem}{Lemma}[section]
\newtheorem{cor}{Corollary}[section]
\theoremstyle{definition}
\newtheorem{rem}{Remark}[section]
\newtheorem{ex}{Example}[section]

\newcommand{\interior}{\operatorname{int}}
\newcommand{\frob}{\mathcal{P}}
\newcommand{\bl}{\bar{\ell}_d}

\newcommand{\n}{\mathbb{N}}
  
\renewcommand {\>}{\right\rangle}  
\newcommand {\norma}[1]{\left\|#1\right\|}

\newcommand{\pr}{\mathbb{P}}

\newcommand{\ww}{\mathcal{W}}

\newcommand{\pp}{\mathcal{P}}

\newcommand{\bj}[1]{\mathbf{j}_{#1}}
\newcommand{\bs}[1]{\mathbf{t}_{#1}}
\newcommand{\bt}[1]{\boldsymbol{\theta}_{#1}}

\newcommand{\bbj}[1]{\hat{\mathbf{j}}_{#1}}
\newcommand{\bbs}[1]{\hat{\mathbf{t}}_{#1}}
\newcommand{\bbt}[1]{\hat{\boldsymbol{\theta}}_{#1}}

\newcommand{\bbbj}[1]{\bar{\mathbf{j}}_{#1}}
\newcommand{\bbbs}[1]{\bar{\mathbf{t}}_{#1}}
\newcommand{\bbbt}[1]{\bar{\boldsymbol{\theta}}_{#1}}

\newcommand{\supp}{\operatorname{supp}}

\numberwithin{equation}{section}

\title{On absolute continuity of invariant measures associated with a piecewise-deterministic Markov processes with random switching between flows}

\author{Dawid Czapla, Katarzyna Horbacz and Hanna Wojew\'odka-\'Sci\k{a}\.zko}
\date{}
\begin{document}
\maketitle
\begin{abstract}
We are concerned with the absolute continuity of stationary distributions corresponding to some~piecewise deterministic Markov process, being typically encountered in biological models. The process under investigation involves a deterministic motion punctuated by random jumps, occurring at the jump times of a Poisson process. The post-jump locations are obtained via random transformations of the pre-jump states. Between the jumps, the motion is governed by continuous semiflows, which are switched directly after the jumps. The main goal of this paper is to provide a set of verifiable conditions implying that any invariant distribution of the process under consideration that corresponds to an ergodic invariant measure of the Markov chain given by its post-jump locations has a density with respect to the Lebesgue measure.
\end{abstract}
{\small
}
\vspace{-0.1cm}
\section*{Introduction}
The object of our study is a subclass of piecewise-deterministic Markov processes (PDMPs), somewhat similar to that considered in \cite{b:bakhtin, b:benaim2, b:benaim1, b:costa, b:traple, b:locherbach}, which plays an important role in biology, providing a mathematical framework for the analysis of gene expression dynamics (cf. \cite{b:mackey_tyran, b:rudnicki_tyran}). Recall that a Markov process may be regarded as belonging to the class of PDMPs whenever, roughly speaking, its randomness stems only from the jump mechanism and, in particular, it admits no diffusive dynamics. This huge class of processes has been introduced by Davis \cite{b:davis}, and arises naturally in many applied areas, such as population dynamics 
\cite{b:benaim_popul, b:cloez_popul}, neuronal activity \cite{b:pakdaman}, excitable membranes \cite{b:riedler}, storage modelling \cite{b:boxma} or internet traffic \cite{b:graham}.

The process considered in this paper is an instance of that introduced in \cite{b:czapla_erg}, and further examined in \cite{b:cz_kub} (cf. also \cite{b:czapla_lil_cont, b:czapla_lil, b:czapla_clt}). More specifically, we study a Markov process $\{(Y(t),\xi(t))\}_{t\geq 0}$ evolving on $Y\times I$, where~$Y$ is a closed subset of $\mathbb{R}^d$ (but not necessarily bounded, in contrast to e.g. \cite{b:benaim1}), and $I$ is a finite set. It is assumed that the process involves a deterministic motion punctuated by random jumps, appearing at random moments $\tau_1<\tau_2<\ldots$, coinciding with the jump times of a homogeneous Poisson process. The underlying random dynamical system can be described in terms of a finite collection $\{S_i:\,i\in I\}$ of semiflows, acting from $[0,\infty)\times Y$ to~$Y$, and an arbitrary family $\{w_{\theta}:\,\theta\in \Theta\}$ of transformations from $Y$ into itself. In the main part of the paper, we assume that $\Theta$ is either an interval in $\mathbb{R}$ or a finite set. Between any two consecutive jumps, the evolution of the first coordinate $Y(\cdot)$ is driven by a semiflow $S_i$, where $i$ is the value of the second coordinate $\xi(\cdot)$. The latter is constant on each time interval between jumps and it is randomly changed right after the jump, depending on the current states of both coordinates. Moreover, the post-jump location of the first coordinate after the $n$th jump, i.e. $Y(\tau_n)$, is obtained as a result of transforming the pre-jump state $Y(\tau_n-)$, using a map $w_{\theta}$, where the index $\theta$ is randomly drawn from $\Theta$, depending on this state. It is worth noting here that such transformations are not present e.g. in the models discussed in \cite{b:bakhtin,  b:benaim2, b:benaim1, b:costa}, where the jumps are only related to the semiflow changes. Consequently, the first coordinate of the process can be shortly expressed as
$$
Y(t)=
\begin{cases}
S_{\xi(t)}(t-\tau_n,Y(\tau_n)) &\mbox{for  } t\in [\tau_n,\tau_{n+1}),\;n\in\mathbb{N}_0,\\
w_{\theta_{n+1}}(Y(\tau_{n+1}-))  &\mbox{for } t=\tau_{n+1}, \;n\in\mathbb{N}_0,
\end{cases}
$$
where $\tau_0=0$, and $\{\theta_n\}_{n\in\n}$ is an appropriate sequence of random variables with values in $\Theta$. In our study, a significant role will be also played by the discrete-time Markov chain $\{(Y_n,\xi_n)\}_{n\in\n}$ defined by
$$Y_n:=Y(\tau_n),\;\;\xi_n:=\xi(\tau_n)\quad\text{for}\quad n\in\n_0,$$
to which we will further refer as to the chain given by the post-jump locations.

In \cite[Theorem 4.1]{b:czapla_erg} (cf. also \cite{b:cz_kub}), we have provided a set of tractable conditions implying that the chain $\{(Y_n,\xi_n)\}_{n\in\n}$ is geometrically ergodic in the the Fortet--Mourier metric (also known as the dual-bounded Lipisithz distance; see \cite{b:las_frac}), which induces the topology of weak convergence of probability measures (see \cite{b:dudley}). This means that the chain possesses a unique, and thus ergodic, stationary distribution, and, for any initial state, the distribution of the chain (at consecutive time points) converges weakly to the stationary one at a geometric rate with respect to the above-mentioned distance. Moreover, we have established a one-to-one correspondence between invariant distributions of that chain and those of the process $\{(Y(t),\xi(t))\}_{t\geq 0}$ (see \hbox{\cite[Theorem 4.4]{b:czapla_erg}}). This has led us to the conclusion that the aforementioned conditions guarantee the existence and uniqueness of a stationary distribution for the PDMP as well. Although not relevant here, it is worth mentioning that the aforesaid results are valid in a more general setting than the one given above; namely, it is enough to require that $Y$ is a Polish metric space, and $\Theta$ is an arbitrary topological measurable space endowed with a finite measure.

The main goal of the present paper is to provide certain verifiable conditions that would imply the absolute continuity of all the stationary distributions of the PDMP $\{(Y(t),\xi(t)\}_{t\geq 0}$ which correspond to ergodic stationary distributions of the chain $\{(Y_n,\xi_n)\}_{n\in\n}$ (see Theorem \ref{main:1}). The~absolute continuity is understood here to hold with respect to the product measure $\bl$ of the \hbox{$d$-dimensional} Lebesgue measure and the counting measure on $I$. As we shall see in Theorem~\ref{cor:2}(ii), the problem reduces, in fact, to examining the invariant distributions of the Markov chain given by the post-jump locations. 

Simultaneously, it should be emphasized that the hypotheses of the above-mentioned\linebreak \cite[Theorem 4.4]{b:czapla_erg} do not ensure that the unique (and thus ergodic) stationary distribution of the chain $\{(Y_n,\xi_n)\}_{n\in\n}$ (or that of the continuous-time process) is absolutely continuous. The simplest example illustrating this claim is a system including only one transformation $w_1\equiv 0$, for which the Dirac measure at $0$ is a unique stationary distribution. 

On the other hand, it is well known and not hard to prove that, whenever the transition operator of a Markov chain preserves the absolute continuity of measures, then any ergodic stationary distribution of the chain (or, in other words, any ergodic invariant probability measure of the transition operator) must be either singular or absolutely continuous (see \cite[Lemma 2.2 with Remark 2.1]{b:las_ac} and cf. \cite[Theorem 6]{b:bakhtin}). As~will be clarified later (in Lemma~\ref{prop:1}), this is the case for the chain $\{(Y_n,\xi_n)\}_{n\in\n}$ if, for instance, all the transformations $w_{\theta}$ and $S_i(t,\cdot)$ are non-singuar with respect to the Lebesgue measure.  Yet, as shown in Example \ref{ex:2}, even under this assumption, the conditions imposed in~\cite{b:czapla_erg} do not guarantee that a unique invariant distribution of the chain and, thus, that of the PDMP, is absolutely continuous. It should be also stressed that, in general, the singularity of some of the transformations $w_{\theta}$ does not necessarily exclude the absolute continuity of invariant measures as well (see e.g. \cite{b:locherbach}).

Obviously, the above-mentioned absolute continuity/singularity dichotomy significantly simplifies the analysis, since, in such a setting, we only need to guarantee that the continuous part of a given ergodic invariant  distribution of $\{(Y_n,\xi_n)\}_{n\in\n}$, say $\mu_*$, is non-trivial. One way to achieve this is to provide the existence of an open $\bl$-small set (in the sense of \cite{b:meyn}) that is uniformly accessible from some measurable subset of $Y\times I$ with positive measure~$\mu_*$ in a specified number of steps \hbox{(see Proposition~\ref{prop:prop2})}.

Following ideas of \cite{b:benaim1}, we show (in Lemma~\ref{lem:rank}) that the existence of an open small set, including a given point $(y_0,j_0)$, can be accomplished by assuming that, for some $n\geq d$ and certain ``admissible'' paths $(j_1,\ldots,j_{n-1})\in I^{n-1}$, $(\theta_1,\ldots,\theta_n)\in\Theta^n$, the composition $$(0,\infty)^n\ni(t_1,\ldots,t_n)\mapsto w_{\theta_n}(S_{j_{n-1}}(t_n,\ldots w_{\theta_1}(S_{j_0}(t_1,y_0))\ldots))$$
has at least one regular point (at which it is a submersion). This requirement is similar in nature to that employed e.g. in \hbox{\cite{b:bakhtin, b:benaim1, b:rudnicki_tyran}}, involving the so-called cumulative flows, which can be usually checked by using a H\"{o}rmander's type condition (see \cite[Theorems 4 and 5]{b:bakhtin}). Furthermore, if~the chain is asymptotically stable, i.e., it admits a unique invariant probability measure to which the distribution of the chain converges weakly, \hbox{independently} of the initial state (which is the case, e.g., under the hypotheses employed in~\cite{b:czapla_erg}), and $(y_0,j_0)$ belongs to the support of~$\mu_*$, then the Portmanteau theorem \hbox{(\cite[Theorem 2.1]{b:bil})} ensures that every open neighbourhood of $(y_0,j_0)$ is uniformly accessible from some other (sufficiently small) neighbourhood of this point with positive measure~$\mu_*$ in a given number of steps (cf. Corollary~\ref{cor:3}). In~general, the latter may, however, be difficult to verify directly, and the argument works only if the chain is asymptotically stable. Therefore, we also propose a more practical condition ensuring the accessibility (cf. Lemma \ref{lem:3}), which concerns the above-specified compositions of $w_{\theta}$ and $S_j$.

Finally, let us drawn attention to the special case where  $S_i(t,y):=y$ for every $i\in I$ (which is, however, out of the scope of this paper). In this case, we have $Y_{n+1}=w_{\theta_{n+1}}(Y_n)$ for any $n\in\n_0$, and thus $\{Y_n\}_{n\in\n_0}$ can be viewed as~a random iterated function system (IFS in short) with place-dependent probabilities (also called a learning system; cf. \cite{b:las_frac, b:las_low, b:szarek_ifs}). The results in \cite{b:szarek_generic} (cf. also \cite{b:las_ac}) show that for most (in the sense of Baire category) such systems the corresponding invariant measures are singular, at least in the case where $\Theta$ is finite and $Y$ is a compact convex subset of $\mathbb{R}^d$. More precisely, it has been proved that asymptotically stable IFSs with singular invariant measures constitute a residual subset of the family of all Lipschitzian IFSs enjoying some additional property that somehow links the Lipschitz constants of $w_{\theta}$ with the associated probabilities.

The outline of the paper is as follows. In Section \ref{sec:1}, we introduce notation and basic definitions regarding Markov operators acting on measures, as well as we give a proof of the aforementioned result regarding the absolute continuity/singularity dichotomy for their ergodic invariant measures. \hbox{Section}~\ref{sec:2} provides a detailed description of the model under study. The main results are established in Section \ref{sec:4}, which is divided into two parts. Section~\ref{sec:41} contains an interpretation of the dichotomy criterion in the given framework and a significant conclusion on the mutual dependence between the absolute continuity of stationary distributions of the chain given by the post-jump locations and the corresponding invariant distributions of the PDMP. Here we also state a  general key observation, linking the absolute continuity of the ergodic invariant distributions of $\{(Y_n,\xi_n)\}_{n\in\n}$ with the existence of a suitable open $\bl$-small set. Further, in Section \ref{sec:42}, we provide some testable conditions implying the existence of such a set and, therefore, guaranteeing the absolute continuity of the invariant measures under consideration.  \hbox{Section} \ref{sec:3} contains the statement of \hbox{\cite[Theorem 4.1]{b:czapla_erg}}, providing the exponential ergodicity of the chain $\{(Y_n,\xi_n)\}_{n\in\n}$ (and hence the existence and uniqueness of a stationary distribution for the PDMP). Some remarks and examples related to our main result are given in Section~\ref{sec:5}.


\section{Preliminaries}\label{sec:1}
Let $(E,\rho)$ be an arbitrary separable metric space, endowed with the Borel $\sigma$-field $\mathcal{B}(E)$. Further, let $\mathcal{M}_{fin}(E)$ be the set of all finite non-negative Borel measures on $E$, and let $\mathcal{M}_{prob}(E)$ stand for the subset of $\mathcal{M}_{fin}(E)$ consisting of all probability measures. Moreover, by $\mathcal{M}^1_{prob}(E)$ we will denote the set of all measures $\mu\in \mathcal{M}_{prob}(E)$ with finite first moment, i.e. satisfying
$$\int_E \rho(x,x^*)\,\mu(dx)<\infty\quad \text{for some}\quad x^*\in E.$$

Now, suppose that we are given a $\sigma$-finite Borel measure $m$ on $E$. Then, a $\sigma$-finite Borel measure $\mu$ on $E$ is  called \emph{absolutely continuous} with respect to $m$, which is denoted by $\mu\ll m$, whenever
$$\mu(A)=0\quad\text{for any}\quad A\in \mathcal{B}(E)\quad\text{such that}\quad m(A)=0.$$
Let $\mathcal{L}^1(E,m)$ denote the space of all Borel measurable and $m$-integrable functions from $E$ to $\mathbb{R}$, identified, as usual, with the corresponding quotient space under the relation of $m$-a.e. equality. Then, by the Radon-Nikodym theorem, $\mu\ll m$  can be equivalently characterized by saying that there is a unique function $f^{\mu}\in\mathcal{L}^1(E,m)$, usually denoted by $d\mu/dm$, such that
$$\mu(A)=\int_A f^{\mu}(x)\,m(dx),\;\;\quad A\in\mathcal{B}(E).$$

The measure $\mu$~is said to be \emph{singular} with respect to $m$, which is denoted by $\mu\perp m$, if there exists a set $F\in\mathcal{B}(E)$ such that 
$$\mu(F)=0\quad\text{and}\quad m(E\backslash F)=0.$$

It is well-known that, due to the Lebesgue decomposition theorem, any $\sigma$-finite Borel measure $\mu$ can be uniquely decomposed as
$$\mu=\mu_{ac}+\mu_s,\quad\text{so that}\quad \mu_{ac}\ll m\quad\text{and}\quad \mu_s\perp m.$$
With regard to the definitions given above, we will use the following notation:
\begin{gather*}
\mathcal{M}_{ac}(E,m):=\{\mu\in\mathcal{M}_{fin}(E):\,\mu\ll m\},\\
\mathcal{M}_{sig}(E,m):=\{\mu\in\mathcal{M}_{fin}(E):\,\mu\perp m\}.
\end{gather*}

Let us now briefly recall the concept of \emph{Frobenius-Perron operator}, which will be used in the analysis that follows. For this aim, suppose that we are given a Borel measurable transformation $S:E\to E$ that is \emph{non-singular} with respect to $m$, i.e.
$$m(S^{-1}(A))=0\quad\text{for any}\quad A\in\mathcal{B}(E)\quad\text{satisfying}\quad m(A)=0.$$
The non-singularity condition assures that, if $\mu\in\mathcal{M}_{ac}(E,m)$, and $\mu_S$ is defined by
$$\mu_S(A):=\mu(S^{-1}(A))\quad\text{for any}\quad A\in\mathcal{B}(E),$$
then $\mu_S\in\mathcal{M}_{ac}(E,m)$. This observation allows one to define a non-negative linear operator $\frob_S: \mathcal{L}^1(E,m)\to \mathcal{L}^1(E,m)$ in such a way that 
$$\frob_S\left(\frac{d\mu}{dm}\right)=\frac{d\mu_S}{dm}\quad\text{for any}\quad\mu\in\mathcal{M}_{ac}(E,m),$$
which, in other words, means that
\begin{equation}
\label{e:frob}
\int_A \frob_S f(x)\,m(dx)=\int_{S^{-1}(A)}f(x)\,m(dx)\quad\text{for any}\quad A\in\mathcal{B}(E),\,f\in\mathcal{L}^1(E,m).\end{equation}
Such an operator $\frob_S$ is commonly known as a Frobenius--Perron operator.

Now, we shall recall several basic definitions from the theory of Markov operators, which will be used throughout the paper. A function $P:E\times\mathcal{B}(E)\rightarrow \left[0,1\right]$ is called a \emph{stochastic kernel} if for each \hbox{$A\in\mathcal{B}(E)$}, \hbox{$x\mapsto P(x,A)$} is a measurable map on $E$, and for each $x\in E$, $A\mapsto P(x,A)$ is a probability Borel measure on $\mathcal{B}(E)$. For any given stochastic kernel $P$, we can consider the corresponding operator \hbox{$P:\mathcal{M}_{fin}(E)\to\mathcal{M}_{fin}(E)$}, acting on measures, given by
\begin{equation} \label{regp} P\mu(A)=\int_{X} P(x,A)\,\mu(dx)\;\;\;\mbox{for}\;\;\; \mu\in\mathcal{M}_{fin}(E),\;A\in \mathcal{B}(E).\end{equation}
Such an operator is usually called \emph{a regular Markov operator}. For notational simplicity, we use here the same symbol for the stochastic kernel and the corresponding Markov operator. This slight abuse of notation will not, however, lead to any confusion.

We say that the operator $P$ is \emph{Feller} (or that it enjoys the \emph{Feller property}) whenever the map~$x\mapsto \<f,P\delta_x\>$ is continuous for any bounded continuous function $f:E\to \mathbb{R}$.

A measure $\mu_*\in\mathcal{M}_{fin}(E)$ is called \emph{invariant} for the Markov operator $P$ (or, simply, $P$-invariant) if $P\mu_*=\mu_*$. If there exists a unique $P$-invariant measure $\mu_*\in\mathcal{M}_{prob}(E)$ such that, for any \hbox{$\mu\in\mathcal{M}_{prob}(E)$}, the sequence $\{P^n\mu\}_{n\in\n}$ is weakly convergent to $\mu_*$, then the operator $P$ is said to be \emph{asymptotically stable}. Let us recall here that a sequence $\{\mu_n\}_{n\in\n}\subset\mathcal{M}_{fin}(X)$ is said to be \emph{weakly convergent} to $\mu\in \mathcal{M}_{fin}(X)$ whenever $$\int_X f\,d\mu_n \to \int_X f\,d\mu,\;\; \text{as}\;\; n\to \infty,$$ for any bounded continuous function $f:E\to\mathbb{R}$

\begin{rem}\label{rem:asymp}
Suppose that $P$ is a regular Markov--Feller operator, and that there exists a measure $\mu_*\in\mathcal{M}_{prob}(E)$ such that $\{P\delta_x\}_{n\in\n}$ is weakly convergent to $\mu_*$ for any $x\in E$. Then $P$ is asymptotically stable.
\end{rem}
\begin{proof}
First of all, note that, due to the Feller property, $P:\mathcal{M}_{prob}(E)\to \mathcal{M}_{prob}(E)$ is continuous in the topology of weak convergence of measures. Taking this into account, we infer that
$$P\mu_*=P(\lim_{n\to\infty} P^n \delta_x) = \lim_{n\to\infty} P^{n+1} \delta_x=\mu_*\quad \text{(with any $x\in E$)},$$
which shows that $\mu_*$ is $P$-invariant. Moreover, using the assumption of the weak convergence of $\{P^n \delta_x\}_{n\in\n}$ (for any $x\in E$) and the Lebesgue's dominated convergence theorem we can simply conclude that $\{P^n \mu\}_{n\in\n}$ converges weakly to $\mu_*$ for any $\mu\in\mathcal{M}_{prob}(E)$. This, in turn, proves that~$\mu_*$ is a unique invariant probability measure for $P$.
\end{proof}

An invariant probability measure $\mu_*\in\mathcal{M}_{prob}(E)$ is said to be \emph{ergodic} with respect to $P$ (or $P$-ergodic) whenever $\mu_*(A)\in\{0,1\}$ for any $A\in\mathcal{B}(E)$ satisfying
$$P(x,A)=1 \quad\text{for}\quad \mu_*\text{ - a.e.}\quad x\in A.$$

It is well-known (see e.g. \cite[Corollary 7.17]{b:douc}) that, if $\mu_*$ is a unique invariant probability measure for $P$, then it must be ergodic. Moreover, according to \cite[Theorem 19.25]{b:aliprantis}, the $P$-ergodic measures are precisely the extreme points of the set of all $P$-invariant probability measures.

\begin{rem}\label{rem1}
If $\mu_*$ is an ergodic invariant measure of $P$, then it cannot be a sum of two distinct non-zero \hbox{$P$-invariant} measures. To see this, suppose that $\mu_*=\mu_1+\mu_2$ for certain non-trivial invariant measures $\mu_1,\mu_2\in\mathcal{M}_{fin}(E)$, and let $\alpha_i:=\mu_i(E)$ for $i=1,2$. Then $\alpha_1+\alpha_2=1$, and $\widetilde{\mu}_i:=\mu_i/\alpha_i$, $i=1,2$, are invariant probability measures for $P$. Since $\mu_*=\alpha_1\widetilde{\mu}_1+\alpha_2\widetilde{\mu}_2$, and $\mu_*$ is an extreme point of the set of $P$-invariant probability measures, we deduce that $\mu_1=\mu_2$.
\end{rem}

The foregoing observation leads to a simple, but extremely useful conclusion regarding the dichotomy between absolute continuity and singularity of $P$-ergodic measures, which can be found e.g.~in \cite[Lemma 2.2, Remark 2.1]{b:las_ac}. Here we provide the proof of this result just for the self-containedness of the paper.
\begin{lem}\label{lem:abs}
Suppose that $P:\mathcal{M}_{fin}(E)\to \mathcal{M}_{fin}(E)$ is a regular Markov operator which preserves absolute continuity of measures, i.e. $P(\mathcal{M}_{ac}(E,m))\subset \mathcal{M}_{ac}(E,m)$. Then, every ergodic invariant probability measure of $P$ is either absolutely continuous or singular with respect to $m$. 
\end{lem}
\begin{proof}
Let $\mu_*\in\mathcal{M}_{prob}(E)$ be an ergodic ergodic $P$-invariant measure. By virtue of the Lebesgue decomposition theorem we can write
\begin{equation}
\label{dec}
\mu_*=\mu_{ac}+\mu_s,
\end{equation}
where $\mu_{ac}\in \mathcal{M}_{ac}(E,m)$ and $\mu_s\in \mathcal{M}_{sig}(E,m)$ are uniquely determined by $\mu_*$. Consequently, it now follows that
$$
P\mu_*=P\mu_{ac}+P\mu_s.
$$
From the principal assumption of the lemma we know that $P\mu_{ac}\in \mathcal{M}_{ac}(E,m)$. Further, using the invariance of $\mu_*$, we also get $\mu_*=P\mu_{ac}+P\mu_s$. Taking the absolutely continuous part of each side of this equality gives
$$\mu_{ac}=P\mu_{ac}+(P\mu_s)_{ac},$$
which, in particular, implies that
$$\mu_{ac}(E)=\mu_{ac}(E)+(P\mu_s)_{ac}(E).$$
Hence $(P\mu_s)_{ac}\equiv 0$, and thus $P\mu_s\in \mathcal{M}_{sig}(E,m)$. From the identity
$$\mu_{ac}+\mu_s=\mu_*=P\mu_{ac}+P\mu_s$$
and the uniqueness of the Lebesgue decomposition it now follows that both measures $\mu_{ac}$ and $\mu_s$ are invariant for $P$. Finally, taking into account \eqref{dec} and the fact that $\mu_{ac}\neq \mu_s$, we can apply Remark \ref{rem1} to conclude that at least one of the measures $\mu_{ac}$, $\mu_s$ must be trivial, which gives the desired conclusion.
\end{proof}
For any given $E$-valued time-homogeneous Markov chain $\{\Phi_n\}_{n\in\n_0}$, defined on some probability space $(\Omega,\mathcal{F},\mathbb{P})$, the stochastic kernel $P(\cdot,*)$ satisfying
$$P(x,A)=\pr(\Phi_{n+1}\in A\,|\,\Phi_n=x)\quad\text{for any}\quad x\in E,\,A\in\mathcal{B}(E),\,n\in\n_0$$
is called a one-step transition law (or a transition probability kernel) of this process. Obviously, in this case, the Markov operator defined by \eqref{regp} describes the evolution of the distributions $\mu_n(\cdot):=\pr(\Phi_n\in\cdot)$, $n\in\n_0$, that is, $\mu_{n}=P\mu_{n-1}$ for any $n\in\n$. In this connection, an invariant probability measure of $P$ is called a stationary distribution of the chain.

A family of regular Markov operators $\{P_t\}_{t\geq 0}$ on $\mathcal{M}_{fin}(E)$, generated accordingly to  \eqref{regp}, is called a regular Markov semigroup whenever it constitutes a semigroup under composition with $P_0=\operatorname{id}$ as the unity element. A measure $\nu_*\in\mathcal{M}_{fin}(X)$ is said to be invariant for such a semigroup if $P_t\,\nu_*=\nu_*$ for any $t\geq 0$.

Analogously to the discrete-time case, by the transition law (or a transition semigroup) of a homogeneous continuous-time Markov process $\{\Phi(t)\}_{t\geq 0}$ we mean the family $\{P_t(\cdot,*)\}_{t\geq 0}$ of stochastic kernels satisfying
$$P_t(x,A)=\pr(\Phi(t+s)\in A\,|\,\Phi(s)=x)\quad\text{for any}\quad x\in E,\,A\in\mathcal{B}(E),\,s,t\geq 0.$$
Since $\{P_t(\cdot,*)\}_{t\geq 0}$ satisfies the Chapman--Kolmogorov equation, the family $\{P_t\}_{t\geq 0}$ of Markov operators generated by such kernels is a Markov semigroup, which describes the evolution of the distributions $\mu(t)(\cdot):=\pr(\Phi(t)\in\cdot)$, $t\geq 0$, i.e. $\mu(s+t)=P_t\mu(s)$ for any $s,t\geq 0$. In this context, an invariant probability measure of $\{P_t\}_{t\geq 0}$ is called a stationary distribution of the process $\{\Phi(t)\}_{t\geq 0}$.

\section{Description of the model} \label{sec:2}
Let us now present a formal description of the investigated model (originating from \cite{b:czapla_erg}), which has already been briefly discussed in the introduction. Recall that such a system can be viewed as a~PDMP evolving through random jumps, which arrive one by one (at random time points~$\tau_n$) in exponentially distributed time intervals. The parameter of the exponential distribution, determining the jump rate, will be denoted by $\lambda$. The deterministic evolution of the process will be governed by a finite number of continuous semiflows, randomly switched at the jump times.

Let $Y$ be a Polish metric space, endowed with the Borel $\sigma$-field $\mathcal{B}(Y)$, and let $\mathbb{R}_+:=[0,\infty)$. Further, suppose that we are given a finite collection $\{S_i:\,i\in I\}$ of continuous semiflows, where $I=\{1,\ldots, N\}$ and $S_i:\mathbb{R}_+\times Y\to Y$ for any $i\in I$. The semiflows will be switched at the jump times according to a matrix $\{\pi_{ij}:\,i,j\in I\}$ of place-dependent continuous probabilities $\pi_{ij}:Y\to[0,1]$, satisfying
$$\sum_{j\in I} \pi_{ij}(y)=1\quad\text{for any} \quad i\in I,\,y\in Y.$$

Moreover, let $\Theta$ is an arbitrary topological space equipped with a finite Borel measure $\vartheta$, and let $\{w_{\theta}:\,\theta\in\Theta\}$ be an arbitrary family of transformations from $Y$ to itself, such that the map $(y,\theta)\mapsto w_{\theta}(y)$ is continuous. These transformations will be related to the post-jump locations of the process; more specifically, if the system is in the state $y$ just before a jump, then its position directly after the jump should be $w_{\theta}(y)$ with some randomly selected $\theta\in\Theta$. The choice of $\theta$ depends on the current state $y$ and is determined by a probability density function $\theta \mapsto p_{\theta}(y)$ such that $(\theta,y)\mapsto p_{\theta}(y)$ is continuous.

The state space of the model under investigation will be $X:=Y\times I$, endowed with the product topology. For any given probability Borel measure $\mu$ on $X$, we first introduce a discrete-time \hbox{$X$-valued} stochastic process $\{\Phi_n\}_{n\in\n_0}$ of the form $\Phi_n=(Y_n,\xi_n)$ with initial distribution $\mu$, defined on a~suitable probability space endowed with a probability measure $\mathbb{P}_{\mu}$, so that
\begin{equation}
\label{e:def_chain}
Y_n=w_{\theta_n}(S_{\xi_{n-1}}(\Delta\tau_n,Y_{n-1}))\quad\text{with}\quad\Delta\tau_n:=\tau_n-\tau_{n-1}\quad\text{for any}\quad n\in\mathbb{N},
\end{equation}
where $\{\tau_n\}_{n\in\n_0}$, $\{\xi_n\}_{n\in\n}$ and $\{\theta_n\}_{n\in\n}$ are the sequences of random variables with values in $\mathbb{R}_+$, $I$ and $\Theta$, respectively, constructed in such~a way that $\tau_0=0$, $\tau_n\to\infty$ (as $n\to\infty$) $\pr_{\mu}$-a.s, and, for every $n\in\n$, we have
\begin{gather*}
\pr_{\mu}(\Delta\tau_n\leq t\,|\;\mathcal{G}_{n-1})=1-e^{-\lambda t}\quad\text{whenever}\quad t\geq 0,\\
\pr_{\mu}(\xi_n=j\,|\,\xi_{n-1}=i,\,Y_{n}=y;\;\mathcal{G}_{n-1})=\pi_{ij}(y)\quad\text{for any}\quad y\in Y,\; i,j\in I,\\
\pr_{\mu}(\theta_n\in D\,|\,S(\Delta\tau_n,Y_{n-1})=y;\,\mathcal{G}_{n-1})=\int_D p_{\theta}(y)\,\vartheta(d\theta)\quad\text{for any}\quad D\in\mathcal{B}(\Theta),\;y\in Y,
\end{gather*}
where $\mathcal{G}_{n-1}$ is the $\sigma$-field generated by the variables $Y_0$, $\tau_1,\ldots,\tau_{n-1}$, $\xi_1,\ldots,\xi_{n-1}$ and $\theta_1,\ldots,\theta_{n-1}$.

Under the assumption that $\xi_n$, $\theta_n$ and $\Delta\tau_n$ are conditionally independent given $\mathcal{G}_{n-1}$ for any $n$, it is easy to check that $\{\Phi_n\}_{n\in\n_0}$ is a time-homogeneous Markov chain with transition probability kernel $P:X\times \mathcal{B}(X)\to [0,1]$ of the form
\begin{align}
\begin{split}
\label{e:kernel}
P((y,i),A)&=\pr_{\mu}(\Phi_{n+1}\in A\,|\,\Phi_n=(y,i))\\
&=\sum_{j\in I}\int_{\Theta}\int_0^{\infty}\lambda e^{-\lambda t}\mathbbm{1}_A(w_{\theta}(S_i(t,y)),j)\pi_{ij}(w_{\theta}(S_i(t,y)))\,p_{\theta}(S_i(t,y))\,dt\,\vartheta(d\theta)
\end{split}
\end{align} 
for $(y,i)\in X$ and $A\in\mathcal{B}(X)$. 

On the same probability space, we can now define an interpolation $\{\Phi(t)\}_{t\geq 0}$ of the chain $\{\Phi_n\}_{n\in\n_0}$ as follows:
$$\Phi(t):=(Y(t),\xi(t))\quad\text{for any}\quad t\geq 0,\vspace{-0.2cm}$$
where
\begin{equation}
\label{def_phi_t}
Y(t):=S_{\xi_n}(t-\tau_n,Y_n)\quad\text{and}\quad\xi(t):=\xi_n,\quad\text{whenever}\quad t\in[\tau_n,\tau_{n+1})\quad\text{for any}\quad n\in\mathbb{N}_0.
\end{equation}
It is easily seen that $\{\Phi(t)\}_{t\geq 0}$ is a time-homogeneous Markov process satisfying $\Phi(\tau_n)=\Phi_n$ for any $n\in\n_0$. By $\{P_t\}_{t\geq 0}$ we will denote the transition semigroup of this process, i.e. 
\begin{equation}
\label{e:kernel_cont}
P_t((y,i),A)=\pr_{\mu}(\Phi(s+t)\in A\,|\,\Phi(s)=(y,i))\quad\text{for any}\quad (y,i)\in X,\;A\in\mathcal{B}(X),\;s,t\geq 0. \vspace{-0.02cm}
\end{equation}

Referring to $P$ and $\{P_t\}_{t\geq 0}$ in our further discussion, we will always mean the Markov operator generated by the transition law of $\{\Phi_n\}_{n\in\n_0}$, given by \eqref{e:kernel}, and the Markov semigroup induced by the transition law of the process $\{\Phi(t)\}_{t\geq 0}$, satisfying \eqref{e:kernel_cont}, respectively. It is worth noting here that, by continuity of functions $S_i$, $w_{\theta}$, $p_{\theta}$ and $\pi_{ij}$, both the operator $P$ and the semigroup $\{P_t\}_{t\geq 0}$ are Feller (cf. \cite[Lemma 6.3]{b:czapla_erg}).

As mentioned in the introduction, a set of directly testable conditions for the existence and uniqueness of $P$-invariant probability measures (which, simultaneously, guarantee a form of geometric ergodicity for $P$) has already been established in \cite[Theorem 4.1]{b:czapla_erg}. This theorem will be quoted in Section \ref{sec:3}. The main results of the paper, concerning the absolute continuity of ergodic invariant measures for $P$, presented in Section \ref{sec:4}, will be derived by assuming a priori that such measures exist.

Let us also recall that, by virtue of \cite[Theorem 4.4]{b:czapla_erg}, there is a one-to-one correspondence between invariant probabability measures of the operator $P$ and those of the semigroup $\{P_t\}_{t\geq 0}$. Moreover, such a correspondence can be expressed explicitly, using the Markov operators $G$ and $W$ induced by the stochastic kernels of the form
\begin{equation}
\label{def:g}
G((y,i),A)=\int_0^{\infty} \lambda e^{-\lambda t} \mathbbm{1}_A(S_i(t,y),i)\,dt,
\end{equation}
\begin{equation}
\label{def:w}
W((y,i),A)=\sum_{j\in I} \int_{\Theta} \mathbbm{1}_A(w_{\theta}(y),j)\pi_{ij}(w_{\theta}(y))p_{\theta}(y)\,\vartheta(d\theta)
\end{equation}
for any $(y,i)\in X$ and $A\in\mathcal{B}(X)$. More precisely, the following holds:
\begin{thm}[\textrm{\cite[Theorem 4.4]{b:czapla_erg}}]\label{thm:one-to-one}
Let $P$ and $\{P_t\}_{t\geq 0}$ denote the Markov operator and the Markov semigroup induced by \eqref{e:kernel} and \eqref{e:kernel_cont}, respectively.
\begin{itemize}
\item[(i)]\phantomsection\label{cnd:1r} If $\mu_*\in\mathcal{M}_{prob}(X)$ is an invariant measure of $P$, then $G\mu_*$ is an invariant  measure of~$\{P_t\}_{t\geq 0}$ and $WG\mu_*  = \mu_*$.
\item[(ii)]\phantomsection\label{cnd:2r} If $\nu_*\in\mathcal{M}_{prob}(X)$ is an invariant measure of $\{P_t\}_{t\geq 0}$, then \hbox{$W\nu_*$} is an invariant  measure of~$P$ and $GW\nu_*=\nu_*$.
 \end{itemize}
\end{thm}

\section{Main results}\label{sec:4}
Throughout the remainder of the paper, we assume that $Y$ is a closed subset of $\mathbb{R}^d$ (endowed with the Euclidean norm $\|\cdot\|$) such that $\interior Y\neq 0$, and we write $\ell_d$ for the $d$-dimensional Lebesgue measure restricted to $\mathcal{B}(Y)$. Moreover, by $\bl$ we denote the product measure $\ell_d\otimes m_c$ on \hbox{$X=Y\times I$}, where $m_c$ is the counting measure on $I$. The latter can be therefore expressed as $m_c(J)=\sum_{j\in I}\delta_j(J)$ for any $J\subset I$, where $\delta_j$ stands for the Dirac measure at $j$. Our aim is to find conditions ensuring the absolute continuity (with respect to $\bar{\ell}_d$) of ergodic invariant probability measures of the operator $P$, if any exist, and the corresponding invariant measures of the semigroup $\{P_t\}_{t\geq 0}$.

\subsection{Singularity/absolute continuity dichotomy of ergodic $P$-invariant measures} \label{sec:41}
We begin our analysis with a simple observation regarding the case in which the operator $P$ preserves the absolute continuity.
\begin{lem}\label{prop:1}
Suppose that, for any $\theta\in\Theta$, $k\in I$ and $t\geq 0$, the transformations $w_{\theta}$ and $S_k(t,\cdot)$ are non-singular with respect to $\ell_d$. Then the Markov operator $P$ induced by \eqref{e:kernel} satisfies $$P\left(\mathcal{M}_{ac}(X,\bl)\right)\subset\mathcal{M}_{ac}\left(X,\bl\right).$$
\end{lem}
\begin{proof}
For any $\theta\in\Theta$, $j\in I$ and $t\geq 0$, let us define $T_{\theta,j,t}:X\to X$ by 
$$T_{\theta,j,t}(y,i):=(w_{\theta}(S_i(t,y)),j)\quad\text{for any}\quad (y,i)\in X.$$
Obviously, each of the transformations $T_{\theta,j,t}$ is then Borel measurable and non-singular with respect to $\bl$. Consequently, for any $\theta\in\Theta$, $j\in I$ and $t\geq 0$, we can consider the Frobenius--Perron operator associated with $T_{\theta,j,t}$, say $\frob_{\theta,j,t}$, which satisfies
$$\int_{X} \mathbbm{1}_A(T_{\theta,j,t}(y,i))f(y,i)\,\bl(dy,di)=\int_A \frob_{\theta,j,t} f(y,i)\,\bl(dy,di)\quad\text{for any}\quad A\in\mathcal{B}(X),\,f\in\mathcal{L}^1(X,\bl).$$
Let $\mu\in\mathcal{M}_{ac}(X,\bl)$, and, for any $(\theta,j,t)\in \Theta\times I\times\mathbb{R}_+$, define
$$f_{\theta,j,t}^{\mu}(y,i):=\pi_{ij}(w_{\theta}(S_i(t,y)))p_{\theta}(S_i(t,y))\frac{d\mu}{d\bar{\ell}_d}(y,i),\quad (y,i)\in X.$$
Since $f_{\theta,j,t}^{\mu}\in\mathcal{L}^1(X,\bl)$, we obtain
\begin{align*}
P\mu(A)&=
\sum_{j\in I}\int_X\int_{\Theta}\int_0^{\infty}\lambda e^{-\lambda t}\mathbbm{1}_A(w_{\theta}(S_i(t,y)),j)\pi_{ij}(w_{\theta}(S_i(t,y)))\,p_{\theta}(S_i(t,y))\,dt\,\vartheta(d\theta)\,\mu(dy,di)\\
&=\sum_{j\in I}\int_{\Theta}\int_0^{\infty}\lambda e^{-\lambda t}\left(\int_X\mathbbm{1}_A(w_{\theta}(S_i(t,y)),j)\pi_{ij}(w_{\theta}(S_i(t,y)))\,p_{\theta}(S_i(t,y))\,\mu(dy,di)\right)\,dt\,\vartheta(d\theta)\\
&=\sum_{j\in I}\int_{\Theta}\int_0^{\infty}\lambda e^{-\lambda t}\left( \int_X\mathbbm{1}_A(T_{\theta,j,t}(y,i)) f_{\theta,j,t}^{\mu}(y,i)\,\bl(dy,di)\right)dt\,\vartheta(d\theta)\\
&=\sum_{j\in I}\int_{\Theta}\int_0^{\infty}\lambda e^{-\lambda t} \left(\int_A\frob_{\theta,j,t}\left( f_{\theta,j,t}^{\mu}\right)(y,i)\,\bl(dy,di)\right)dt\,\vartheta(d\theta)\\
&=\int_A\left(\sum_{j\in I}\int_{\Theta}\int_0^{\infty}\lambda e^{-\lambda t}\frob_{\theta,j,t}\left( f_{\theta,j,t}^{\mu}\right)(y,i)\,dt\,\vartheta(d\theta)\right)\,\bl(dy,di)\quad\text{for any}\quad A\in\mathcal{B}(X).
\end{align*}
It now follows that the map
$$X\ni(y,i)\mapsto\sum_{j\in I}\int_{\Theta}\int_0^{\infty}\lambda e^{-\lambda t}\frob_{\theta,j,t}\left( f_{\theta,j,t}^{\mu}\right)(y,i)\,dt\,\vartheta(d\theta)$$
is a Radon-Nikodym derivative of $P\mu$ with respect to $\bl$, whence $P\mu\in\mathcal{M}_{ac}(X,\bl)$.
\end{proof}


Proceeding similarly as in the proof of Lemma \ref{prop:1}, we shall now show that also the Markov operators~$G$ and $W$, corresponding to \eqref{def:g} and \eqref{def:w}, respectively, preserve the absolute continuity of measures, whenever $S_k(t,\cdot)$ and $w_{\theta}$ are non-singular.

\begin{lem}\label{prop:2}
Suppose that the assumption of Lemma \ref{prop:1} is fulfilled. Then the Markov operators $G$ and $W$, generated by \eqref{def:g} and \eqref{def:w}, respectively, satisfy
$$G\left(\mathcal{M}_{ac}(X,\bl)\right)\subset\mathcal{M}_{ac}\left(X,\bl\right)\quad\text{and}\quad W\left(\mathcal{M}_{ac}(X,\bl)\right)\subset\mathcal{M}_{ac}\left(X,\bl\right).$$
\end{lem}
\begin{proof}
To prove the first inclusion, for any $t\geq 0$, we define $H_t:X\to X$ by setting 
$$H_t(y,i):=(S_i(t,y),i)\quad\text{for any}\quad (y,i)\in X.$$
Such a transformation is then Borel measurable and non-singular with respect to $\bl$. Hence, we can consider the Frobenius--Perron operator associated with $H_t$, say $\frob_t$, which satisfies
$$\int_{X} \mathbbm{1}_A(H_t(y,i))f(y,i)\,\bl(dy,di)=\int_A \frob_{t} f(y,i)\,\bl(dy,di)\;\;\text{for any}\;\; A\in\mathcal{B}(X),\,f\in\mathcal{L}^1(X,\bl).$$
Letting $\mu\in\mathcal{M}_{ac}(X,\bl)$ and putting $h^{\mu}:={d \mu}/{d \bl}$, we then see that, for any $A\in\mathcal{B}(X)$,
\begin{align*}
G\mu(A)&=\int_X\int_0^{\infty} \lambda e^{-\lambda t} \mathbbm{1}_A(S_i(t,y),i)\,dt\,\mu(dy,di)\\
&=\int_0^{\infty} \lambda e^{-\lambda t}\left(\int_X  \mathbbm{1}_A(H_t(y,i))\,h^{\mu}(y,i)\,\bl(dy,di)\right)\,dt\\
&=\int_0^{\infty} \lambda e^{-\lambda t}\left(\int_A \frob_{t} h^{\mu}(y,i)\,\bl(dy,di)\right)\,dt=\int_A\left(\int_0^{\infty} \lambda e^{-\lambda t} \frob_{t} h^{\mu}(y,i)\,dt \right)\bl(dy,di).
\end{align*}
This shows that
$$(y,i)\mapsto \int_0^{\infty} \lambda e^{-\lambda t} \frob_{t} h^{\mu}(y,i)\,dt$$
is a Radon--Nikodym derivative of $G\mu$ with respect to $\bl$, which means that $G\mu\in\mathcal{M}_{ac}(X,\bl)$ and, therefore, shows the first inclusion in the assertion of the lemma. 

The proof of the second inclusion goes similarly. In this case, for any $\theta\in\Theta$ and any $j\in I$, we consider $R_{\theta,j}:X\to X$ given by
$$R_{\theta,j}(y,i)=(w_{\theta}(y),j)\quad\text{for any}\quad (y,i)\in X.$$
Obviously, all the transformations $R_{\theta,j}$ are Borel measurable and non-singular with respect to $\bl$. This observation, as before, enables us to introduce the Frobenius--Perron operator associated with $R_{\theta,j}$, say $\frob_{\theta,j}$, which satisfies
$$\int_{X} \mathbbm{1}_A(R_{\theta,j}(y,i))f(y,i)\,\bl(dy,di)=\int_A \frob_{\theta,j} f(y,i)\,\bl(dy,di)\;\;\text{for any}\;\; A\in\mathcal{B}(X),\,f\in\mathcal{L}^1(X,\bl).$$
Let $\mu\in\mathcal{M}_{ac}(X,\bl)$ and define
$$r^{\mu}_{\theta,j}(y,i):=\pi_{ij}(w_{\theta}(y))p_{\theta}(y)\frac{d \mu}{d \bl}(y,i)\quad\text{for any}\quad (y,i)\in X.$$
Taking into account that $r^{\mu}\in\mathcal{L}^1(X,\bl)$, we now obtain that, for any $A\in\mathcal{B}(X)$,
\begin{align*}
W\mu(A)&= \sum_{j\in I}\int_X \int_{\Theta} \mathbbm{1}_A(w_{\theta}(y),j)\pi_{ij}(w_{\theta}(y))p_{\theta}(y)\,\vartheta(d\theta) \,\mu(dy,di)\\
&=\sum_{j\in I} \int_{\Theta}\left(\int_X \mathbbm{1}_A(w_{\theta}(y),j)\pi_{ij}(w_{\theta}(y))p_{\theta}(y)\,\mu(dy,di)\right)\vartheta(d\theta)\\
&=\sum_{j\in I} \int_{\Theta}\left(\int_X \mathbbm{1}_A(R_{\theta,j}(y,i))r^{\mu}_{\theta,j}(y,i)\,\bl(dy,di)\right)\vartheta(d\theta)\\
&=\sum_{j\in I} \int_{\Theta}\left(\int_A \frob_{\theta,j}(r_{\theta,j}^{\mu})(y,i)\,\bl(dy,di)\right)\vartheta(d\theta)=\int_A\left(\sum_{j\in I}\int_{\Theta}\frob_{\theta,j}(r_{\theta,j}^{\mu})(y,i)\,\vartheta(d\theta) \right)\bl(dy,di).
\end{align*}
Consequently, we now see that the map $$X\ni (y,i)\mapsto \sum_{j\in I}\int_{\Theta}\frob_{\theta,j}(r_{\theta,j}^{\mu})(y,i)\,\vartheta(d\theta)$$
is a Radon-Nikodym derivative of $W\mu$ with respect to $\bl$, which, in turn, gives $W\mu\in\mathcal{M}_{ac}(X,\bl)$ and completes the proof of the lemma.
\end{proof}

Collecting all the results obtained so far, we can state the following theorem:

\begin{thm}\label{cor:2}
Let $P, G, W$ be the Markov operators generated by \eqref{e:kernel}, \eqref{def:g} and \eqref{def:w}, respectively, and let $\{P_t\}_{t\geq 0}$ be the Markov semigroup corresponding to \eqref{e:kernel_cont}. Further, suppose that, for any $\theta\in\Theta$, $k\in I$ and $t\geq 0$, the transformations $w_{\theta}$ and $S_k(t,\cdot)$ are non-singular with respect to the Lebesgue measure $\ell_d$. Then 
\begin{itemize}
\item[(i)] Every ergodic invariant probability measure of $P$ is either absolutely continuous or singular with respect to $\bl$.
\item[(ii)] If $\mu_*, \nu_*\in\mathcal{M}_{prob}(X)$ are invariant probability measures for $P$ and $\{P_t\}_{t\geq 0}$, respectively, which correspond to each other in the manner of Theorem \ref{thm:one-to-one}, that is,
$$\nu_*=G\mu_*\;\;\text{or, equivalently,}\;\; \mu_*=W\nu_*,$$
then the measure $\mu_*$ is absolutely continuous with respect to $\bl$ if and only if so is $\nu_*$. 
\item[(iii)] If $\mu_*, \nu_*\in\mathcal{M}_{prob}(X)$ are the unique invariant probability measures for $P$ and $\{P_t\}_{t\geq 0}$,\linebreak respectively, then $\mu_*$ is absolutely continuous with respect to $\bl$ if and only if so is $\nu_*$.
\end{itemize}
\end{thm}

\begin{proof}
The first statement of the theorem follows immediately from Lemmas \ref{lem:abs} and \ref{prop:1}. The second one is just a summary of Theorem \ref{thm:one-to-one} and Lemma~\ref{prop:2}. Finally, the last assertion is a straightforward consequence of the second one.
\end{proof}

For a given ergodic $P$-invariant probability measure $\mu_*$, Theorem \ref{cor:2} enables us to restrict our inquiry about the absolute continuity of both $\mu_*$ and $G\mu_*$ to the one about the non-triviality of the continuous part of~$\mu_*$. Certain general conditions providing the positive answer to this question are given in the result below. These conditions should be viewed as a starting point for the forthcoming discussion regarding possible restrictions on the component functions of the model that would guarantee the desired absolute continuity.
\begin{prop}\label{prop:prop2}
Let $\mu_*$ be any invariant probability measure of the Markov operator $P$, corresponding to \eqref{e:kernel}. Suppose that there exist open subsets $U,V$ of $Y$ and an index $i\in I$ such that, for some $n\in\n$ and some $\bar{c}>0$,
\begin{equation}
\label{e1:prop2}
P^n(x,B\times \{j\})\geq \bar{c}\,\ell_d(B\cap V)\quad\text{for any}\quad x\in U\times \{i\},\;j\in I\;\;\text{and}\;\;B\in\mathcal{B}(Y).
\end{equation}
Furthermore, assume that there exist a set $\widetilde{X}\in\mathcal{B}(X)$ with $\mu_*(\widetilde{X})>0$, $m\in\mathbb{N}$ and $\delta>0$ such that
\begin{equation}
\label{e2:prop2}
 P^{m}(x,U\times\{i\})\geq \delta\quad\text{for any}\quad x\in \widetilde{X}.
\end{equation}
Then the absolutely continuous part of $\mu_*$ with respect to $\bar{\ell}_d$ is non-trivial. If, additionally, $\mu_*$ is ergodic and the assumption of Theorem \ref{cor:2} is fulfilled, then both $\mu_*$ and $G\mu_*$ are absolutely continuous with respect to $\bar{\ell}_d$.
\end{prop}
\begin{proof}
Let $B\in\mathcal{B}(Y)$ and $j\in I$. Taking into account the invariance of $\mu_*$ and condition \eqref{e1:prop2},  we can write
\vspace{-0.1cm}
\begin{align*}
\mu_*(B\times\{j\})&=P^n\mu_*(B\times\{j\})=\int_X P^n(x,B\times\{j\})\,\mu_*(dx)\\
&\geq \int_{U\times \{i\}}P^n(x,B\times\{j\})\,\mu_*(dx)\geq\bar{c}\,\ell_d(B\cap V) \mu_*(U\times\{i\}).
\end{align*}
Using again the invariance of $\mu_*$, we get
\begin{align*}
\mu_*(B\times\{j\})&\geq\bar{c}\,\ell_d(B\cap V) P^m\mu_*(U\times\{i\})\geq\bar{c}\,\ell_d(B\cap V)\int_{\widetilde{X}} P^{m}(x,U\times\{i\})\,\mu_*(dx). 
\end{align*}
Finally, applying hypothesis \eqref{e2:prop2} gives
$$\mu_*(B\times\{j\})\geq \bar{c}\delta \mu(\widetilde{X})\ell_d(B\cap V),$$
which shows that $\mu_*$ indeed has a non-trivial absolutely continuous part. The second part of the assertion follows immediately from Theorem \ref{cor:2}.
\end{proof}
The assumptions of Proposition \ref{prop:prop2}, referring to an open set $U\times\{i\}$, may be interpreted as follows. Condition \eqref{e1:prop2} says that this set is $(n,\ell_d|_V)$-\emph{small} in the sense of \cite{b:meyn}. According to \eqref{e2:prop2}, it is also \emph{uniformly accessible} from some subset of $X$ with positive measure $\mu_*$ in some specified number of steps.
\vspace{-0.2cm}
\subsection{A criterion on absolute continuity of ergodic invariant measures associated with the model}\label{sec:42}
In this section, as well as in the rest of the paper, we require that $\Theta$ is either a finite set (with~the discrete topology), equipped with the counting measure $\vartheta=\sum_{\theta\in\Theta}\delta_{\theta}$, or an interval in~$\mathbb{R}$ (with the Euclidean topology), endowed with a finite Borel measure $\vartheta$ that is positive on every non-empty open set (e.g., if $\Theta$ is bounded, we can take $\vartheta=\ell_1|_{\mathcal{B}(\Theta)}$).

As mentioned in the introduction, our main goal is to provide a set of tractable conditions for the components of the model, which are sufficient for the absolute continuity of the unique invariant probability measures associated with the Markov operator $P$  and the semigroup $\{P_t\}_{t\geq 0}$, induced by \eqref{e:kernel} and \eqref{e:kernel_cont}, respectively. To do this, we shall need an explicit form of the $n$th-step kernel \hbox{$(x,A)\mapsto P^n(x,A)$}. For this reason, it is convenient to introduce the following piece of notation.

For any $k\in\n$, let $\bj{k}$,~$\bs{k}$, $\bt{k}$ denote $(j_1,\ldots,j_k)\in I^{k}$, \hbox{$(t_1,\ldots,t_k)\in\mathbb{R}_+^k$}, $(\theta_1,\ldots,\theta_k)\in\Theta^k$, respectively. Further, given any $i,j\in I$, we employ the following convention:
$$(i,\bj{k}):=(i,j_1,\ldots,j_{k})\quad\text{and}\quad (i,\bj{k},j):=(i,j_1,\ldots,j_{k},j).$$
Here, for notational consistency, we additionally put $(i,\bj{0}):=i$ and $(i,\bj{0},j):=(i,j)$ if $k=0$. 

With this notation, for any $n\in\mathbb{N}$, $y\in Y$, $j_0\in I$ and $(\bj{n},\bs{n},\bt{n})\in I^n\times \mathbb{R}_+^n\times\Theta^n$, we may~define
\begin{gather*}
\ww_1(y,j_0,t_1,\theta_1):=w_{\theta_1}(S_{j_0}(t_1,y)),\\
\ww_n(y,\,(j_0,\bj{n-1}),\bs{n},\bt{n}):=w_{\theta_n}(S_{j_{n-1}}(t_n,
\ww_{n-1}(y,(j_0,\bj{n-2}),\bs{n-1},\bt{n-1})));\\[0.3cm]
\Pi_1(y,(j_0,j_1),t_1,\theta_1):=\pi_{j_0 j_1}(w_{\theta_1}(S_{j_0}(t_1,y))),\\
\Pi_n(y,(j_0,\bj{n}),\bs{n},\bt{n}):=\Pi_{n-1}(y,(j_0,\bj{n-1}),\bs{n-1},\bt{n-1})\pi_{j_{n-1}j_n}(\ww_n(y,(j_0,\bj{n-1}),\bs{n},\bt{n}));\\[0.3cm]
\pp_1(y,j_0,t_1,\theta_1):=
p_{\theta_1}(S_{j_0}(t_1,y)),\\
\pp_n(y,(j_0,\bj{n-1}),\bs{n},\bt{n})\hspace{-0.1cm}:=\pp_{n-1}(y,(j_0,\bj{n-2}),\bs{n-1},\bt{n-1}) p_{\theta_n}(S_{j_{n-1}}(t_n, \ww_{n-1}(y,(j_0,\bj{n-2}),\bs{n-1},\bt{n-1}))).
\end{gather*}

The $n$th-step transition law of the chain $\{\Phi_k\}_{k\in\n_0}$ can be now expressed as
\begin{align}
\begin{split}
\label{e:nstep}
P^n((y,i),A)&=\sum_{\bj{n}\in I^n}\,\int_{\Theta^n}\int_{\mathbb{R}_+^n}\lambda^n e^{-\lambda(t_1+\ldots+t_n)}\mathbbm{1}_A(\ww_n(y,(i,\bj{n-1}),\bs{n},\bt{n}),j_n)\\
&\quad\times \Pi_n(y,(i,\bj{n}),\bs{n},\bt{n})
\pp_n(y,(i,\bj{n-1}),\bs{n},\bt{n})\,d\bs{n}\,\vartheta^{\otimes n}(d\bt{n})
\end{split}
\end{align}
for any $(y,i)\in X=Y\times I$ and any $A\in\mathcal{B}(X)$, where the symbols $d\bs{n}$ and $\vartheta^{\otimes n}(d\bt{n})$ represent $\ell_n(dt_1,\ldots,dt_n)$ and $(\vartheta \otimes \ldots \otimes \vartheta)(d\theta_1,\ldots,d\theta_n)$, respectively.

In what follows, we shall assume that, for any $\theta\in\Theta$ and any $i\in I$, the maps
$$Y \ni (y_1,\ldots,y_d)=:y\mapsto w_{\theta}(y)\quad\text{and}\quad (0,\infty)\times Y \ni (t,y)\mapsto S_i(t,y)$$
are continuously differentiable with respect to each of the variables $y_k$, $k=1,\ldots,d$, and $t$. 
In the case where $\Theta$ is an interval, we additionally require that the map $\interior \Theta \ni \theta \mapsto w_{\theta}(y)$ is continuously differentiable for any $y\in Y$.

Let $\hat{y}:=(\hat{y}_1,\ldots,\hat{y}_d)\in Y$, $n\in\n$, $i\in I$, $\bbj{n-1}:=(\hat{j}_1,\ldots,\hat{j}_{n-1})\in I^{n-1}$, $\bbt{n}:=(\hat{\theta}_1,\ldots,\hat{\theta}_n)\in\Theta^n$ and $\bbs{n}:=(\hat{t}_1,\ldots,\hat{t}_n)\in (0,\infty)^n$. For any $m\leq n$ and any pairwise different indices $k_1,\ldots,k_m\in \{1,\ldots,n\}$, the Jacobi \hbox{matrix} of the map 
$$(t_{k_1},\ldots,t_{k_m})\mapsto \ww_n(\hat{y},(i,\bbj{n-1}),\bs{n},\bbt{n})\quad\text{with fixed}\quad t_r=\hat{t}_r \quad\text{for}\quad r\in \{1,\ldots,n\}\backslash \{k_1,\ldots,k_m\}$$ 
at point $(\hat{t}_{k_1},\ldots,\hat{t}_{k_m})$ will be denoted by $\partial_{(t_{k_1},\ldots,t_{k_m})} \ww_n(\hat{y},(i,\bbj{n-1}),\bbs{n},\bbt{n})$. More precisely, assuming that $\mathcal{W}_n=(\mathcal{W}_n^{(1)},\ldots, \mathcal{W}_n^{(d)})$, where $\mathcal{W}_n^{(l)}$ takes values in $\mathbb{R}$ for any $l\in\{1,\ldots,d\}$, we put
$$\partial_{(t_{k_1},\ldots,t_{k_m})} \ww_n(y,(i,\bbj{n-1}),\bbs{n},\bbt{n}):=\left[\frac{\partial \ww_n^{(l)}}{\partial t_{k_r}}(\hat{y},(i,\bbj{n-1}),\bbs{n},\bbt{n}) \right]_{\substack{
r\in\{1,\ldots,m\}\\l\in\{1,\ldots,d\}}}.$$
Analogously, we can define $\partial_{(\theta_{l_1},\ldots,\theta_{l_m})} \ww_n(
\hat{y},(i,\bbj{n-1}),\bbs{n},\bbt{n})$ and $\partial_{(y_{p_1},\ldots,y_{p_m})} \ww_n(
\hat{y},(i,\bbj{n-1}),\bbs{n},\bbt{n})$ for any pairwise different $l_1,\ldots,l_m\in\{1,\ldots,n\}$ and $p_1,\ldots,p_m\in\{1,\ldots,d\}$, respectively.

\vspace*{0.05cm}
A key role in our discussion will be played by the following lemma, which provides a tractable condition under which the operator $P$ verifies the first hypothesis of Proposition \ref{prop:prop2}, expressed in~\eqref{e1:prop2}. The proof of this result is based upon ideas found in \cite{b:benaim1}.

\begin{lem}\label{lem:rank}
Let $(\hat{y},i)\in \interior Y\times I$, and suppose that, for some integer $n\geq d$, there exist sequences $\bbs{n}\in(0,\infty)^n$, $\bbt{n}\in \interior \Theta^n$ and, in the case of $n>1$, also $\bbj{n-1}\in I^{n-1}$, such that
\begin{equation}
\label{e:pos}
\pp_n(\hat{y},(i,\bbj{n-1}),\bbs{n},\bbt{n})\,\Pi_n(\hat{y},(i,\bbj{n-1},j),\bbs{n},\bbt{n})>0 \quad\text{for any}\quad j\in I,
\end{equation}
and
\begin{equation}
\label{e:rank}
\operatorname{rank}\,\partial_{\bs{n}} \ww_n(\hat{y},(i,\bbj{n-1}),\bbs{n},\bbt{n})=d.
\end{equation}
Then there is an open neighbourhood $U_{\hat{y}}\subset Y$ of $\hat{y}$ and an open neighbourhood $U_{\hat{w}}\subset Y$ of the point
\begin{equation}
\label{def:hatw}
\hat{w}:=\ww_n(\hat{y},(i,\bbj{n-1}),\bbs{n},\bbt{n})
\end{equation}
such that, for some constant $\bar{c}>0$, we have
\begin{equation}
\label{e:small1}
P^n(x,B\times\{j\})\geq \bar{c}\, \ell_d(B\cap U_{\hat{w}})\quad\text{for any}\quad x\in U_{\hat{y}}\times\{i\},\,B\in\mathcal{B}(Y),\;j\in I.
\end{equation}
\end{lem}

\begin{proof}
According to \eqref{e:rank}, there exist $k_1,\ldots,k_d\in \{1,\ldots,n\}$ such that
$$
\operatorname{det}\,\partial_{(t_{k_1},\ldots,t_{k_d})}\ww_n(\hat{y},\,(i,\bbj{n-1}),\bbs{n},\bbt{n})\neq 0.
$$
Without loss of generality, we may further assume that $(k_1,\ldots,k_d)=(1,\ldots,d)$, i.e.
\begin{equation}
\label{e:det1}
\operatorname{det}\,\partial_{\bs{d}}\ww_n(\hat{y},\,(i,\bbj{n-1}),\bbs{n},\bbt{n})\neq 0.
\end{equation}
In the analysis that follows, given $\bs{n}=(t_1,\ldots,t_n)\in\mathbb{R}_+^n$, we shall write $\bold{t}^{n-d}$ to denote $(t_{d+1},\ldots,t_n)$, so that $\bs{n}=(\bs{d},\bold{t}^{n-d})$.

\textbf{Case I: } Consider first the case where $\Theta$ is finite. For any $y\in Y$, let us introduce the map $\mathcal{R}_y:(0,\infty)^n\to Y\times \mathbb{R}_+^{n-d}\subset\mathbb{R}^n$ given by
$$\mathcal{R}_y(\bs{n}):=(\ww_n(y,(i,\bbj{n-1}),\bs{n},\bbt{n}),\bold{t}^{n-d})\quad\text{for}\quad \bs{n}\in(0,\infty)^n.$$
We can then easily observe that
$$\partial_{\bs{n}}\mathcal{R}_y(\bs{n})=
\left[
\begin{array}{c|c}
\partial_{\bs{d}}\ww_n & \partial_{\bold{t}^{n-d}}\mathcal{W}_n\\
\hline
\mathbf{0}_{d,\, n-d} &I_{n-d}
\end{array}
\right](y,(i,\bbj{n-1}),\bs{n},\bbt{n}),
$$
where $\mathbf{0}_{d,\,n-d}$ and $I_{n-d}$ are the zero matrix of size $d\times (m-d)$ and the identity matrix of order $n-d$, respectively. This yields that
\begin{equation}
\label{e:det_eq}
\operatorname{det}\partial_{\bs{n}}\mathcal{R}_y(\bs{n})=\operatorname{det}\,\partial_{\bs{d}}\ww_n(y,(i,\bbj{n-1}),\bs{n},\bbt{n})\quad\text{for any}\quad \bs{n}\in (0,\infty)^n,\;y\in Y.
\end{equation}

Further, let us define $\mathcal{H}:(0,\infty)^n \times \interior Y \to (Y\times\mathbb{R}_+^{n-d})\times Y$, acting from an open subset of~$\mathbb{R}^{n+d}$ into itself, by
$$\mathcal{H}(\bs{n},y):=(\mathcal{R}_{y}(\bs{n}),y)\quad\text{for any}\quad \bs{n}\in (0,\infty)^n,\;y\in \interior Y.$$
Since the Jacobi matrix of $\mathcal{\mathcal{H}}$ can also be written in a block form, namely
$$\partial_{(\bs{n},y)}\mathcal{H}(\bs{n},y)=
\left[
\begin{array}{c|c}
\partial_{\bs{d}}\ww_n & \partial_{\bold{t}^{n-d}}\mathcal{W}_n \;\; \partial_{y}\mathcal{W}_n\\
\hline
\mathbf{0}_{d,n} &I_{n}
\end{array}
\right](y,(i,\bbj{n-1}),\bs{n},\bbt{n}),
$$
it follows, due to \eqref{e:det1}, that
\begin{equation}
\label{e:det2}
\operatorname{det}\partial_{(\bs{n},y)}\mathcal{H}(\bbs{n},\hat{y})=\operatorname{det}\,\partial_{\bs{d}}\ww_n(\hat{y},\,(i,\bbj{n-1}),\bbs{n},\bbt{n}) \operatorname\neq 0.
\end{equation}
Consequently, by virtue of the local inversion theorem, we can choose an open neighbourhood \hbox{$\widehat{V}_{(\bbs{n},\hat{y})}\subset (0,\infty)^n\times\interior Y$} of  $(\bbs{n},\hat{y})$ so that $\mathcal{H}|_{\widehat{V}_{(\bbs{n},\hat{y})}}:\widehat{V}_{(\bbs{n},\hat{y})}\to \mathcal{H}(\widehat{V}_{(\bbs{n},\hat{y})})$ is a diffeomorphism.\\[-0.1cm] \hbox{Obviously}, $\mathcal{H}(\widehat{V}_{(\bbs{n},\hat{y})})\subset (Y\times \mathbb{R}_+^{n-d})\times Y$.

\vspace*{0.1cm}
If we now define
\begin{equation}
\label{def_t}
\mathcal{T}_n(y,(i,\bbj{n-1},j),\bs{n},\bt{n}):=\lambda^n e^{-\lambda(t_1+\ldots+t_n)}\pp_n(y,(i,\bbj{n-1}),\bs{n},\bt{n})\Pi_n(y,(i,\bbj{n-1},j),\bs{n},\bt{n}),
\end{equation}
then, using \eqref{e:pos} and \eqref{e:det2}, together with continuity of the component functions of the model and the map  $\widehat{V}_{(\bbs{n},\hat{y})}\ni(\bs{n},y)\mapsto \det \partial_{(\bs{n},y)}\mathcal{H}(\bs{n},y)$, we may find an open neighbourhood $\widetilde{V}_{(\bbs{n},\hat{y})}\subset \widehat{V}_{(\bbs{n},\hat{y})}$ of $(\bbs{n},\hat{y})$ such that, for some constant $\tilde{c}>0$,
\begin{equation}
\label{e:pos2a}
\left|\operatorname{det}\partial_{(\bs{n},y)}\mathcal{H}({\bs{n}},y) \right|^{-1}\mathcal{T}_n(y,(i,\bbj{n-1},j),\bs{n},\bbt{n})\geq \tilde{c}\quad\text{for any}\quad (\bs{n},y)\in \widetilde{V}_{(\bbs{n},\hat{y})},\;j\in I.
\end{equation}
Taking into account that, due to \eqref{e:det_eq},  
$$\operatorname{det}\partial_{(\bs{n},y)}\mathcal{H}(\bs{n},y)=\operatorname{det}\,\partial_{\bs{d}}\ww_n(y,(i,\bbj{n-1}),\bs{n},\bbt{n})=\operatorname{det}\partial_{\bs{n}}\mathcal{R}_y(\bs{n}),\vspace{-0.2cm}$$ we then obtain
\begin{equation}
\label{e:pos2}
\left|\operatorname{det}\partial_{\bs{n}}\mathcal{R}_y({\bs{n}}) \right|^{-1}\mathcal{T}_n(y,(i,\bbj{n-1},j),\bs{n},\bbt{n})\geq \tilde{c}\quad\text{for any}\quad (\bs{n},u)\in \widetilde{V}_{(\bbs{n},\hat{y})},\;j\in I.
\end{equation}

Clearly, $\mathcal{H}|_{\widetilde{V}_{(\bbs{n},\hat{y})}}:\widetilde{V}_{(\bbs{n},\hat{y})}\to \mathcal{H}(\widetilde{V}_{(\bbs{n},\hat{y})})$ is also a diffeomorphism, and thus, in particular, the set $\mathcal{H}(\widetilde{V}_{(\bbs{n},\hat{y})})$ is open. Since $((\hat{w},\hat{\bold{t}}^{n-d}),\hat{y})\in \mathcal{H}(\widetilde{V}_{(\bbs{n},\hat{y})})$, where $\hat{w}$ is given by \eqref{def:hatw}, there exist open bounded neighbourhoods $U_{(\hat{w},\,\hat{\bold{t}}^{n-d})}\subset  Y\times \mathbb{R}_+^{n-d}$ and $U_{\hat{y}}\subset Y$ of the points $(\hat{w},\hat{\bold{t}}^{n-d})$ and $\hat{y}$, respectively, with the property that $U_{(\hat{w},\,\hat{\bold{t}}^{n-d})}\times U_{\hat{y}}\subset \mathcal{H}(\widetilde{V}_{(\bbs{n},\hat{y})})$. Let $$V_{(\bbs{n},\hat{y})}:=\mathcal{H}^{-1}(U_{(\hat{w},\,\hat{\bold{t}}^{n-d})}\times U_{\hat{y}}),$$ and, for any $((w,\bold{t}^{n-d}),y)\in U_{(\hat{w},\,\hat{\bold{t}}^{n-d})}\times U_{\hat{y}}$, write $\mathcal{H}^{-1}((w,\bold{t}^{n-d}),y)=(\overline{\mathcal{R}}_y(w,\bold{t}^{n-d}),y)$. Then, it follows immediately that\vspace{0.1cm} $\mathcal{R}_y\left(\overline{\mathcal{R}}_y(w,\bold{t}^{n-d})\right)=(w,\bold{t}^{n-d})$, whence $\overline{\mathcal{R}}_y$ is the continuously differentiable inverse of an appropriate restriction of $\mathcal{R}_y$. More specifically, introducing
$$W(y):=\{\bs{n}\in (0,\infty)^n:\,(\bs{n},y)\in V_{(\bbs{n},\hat{y})} \}\quad\text{for every}\quad y\in U_{\hat{y}},$$
we see that each of these sets is open, and that \hbox{$\mathcal{R}_y|_{W(y)}:W(y)\to U_{(\hat{w},\,\hat{\bold{t}}^{n-d})}$} is a diffeomorphism for any $y\in U_{\hat{y}}$. Obviously, by the definition of $W(y)$, we have 
\begin{equation}
\label{e:wu}
(\bs{n},y)\in V_{(\bbs{n},\hat{y})}\quad\text{whenever}\quad \bs{n}\in W(y),\;y\in U_{\hat{y}}.
\end{equation}

In view of the above, we can choose (independently of $y$) open neighbourhoods $U_{\hat{w}}\subset Y$ and $U_{\hat{\bold{t}}^{n-d}}\subset\mathbb{R}_+^{n-d}$ of $\hat{w}$ and $\hat{\bold{t}}^{n-d}$, respectively, in such a way that
\begin{equation}
\label{e:sub}
U_{\hat{w}}\times U_{\hat{\bold{t}}^{n-d}} \subset U_{(\hat{w},\,\hat{\bold{t}}^{n-d})}= \mathcal{R}_y(W(y))\quad\text{for any}\quad y \in U_{\hat{y}}.
\end{equation}

Now, keeping in mind \eqref{e:nstep}, \eqref{e:pos2} and \eqref{e:wu}, for any $B\in\mathcal{B}(Y)$, $j\in I$ and $y\in U_{\hat{y}}$, we can write
\begin{align*}
P^n((y,i)&, B\times\{j\})\geq \int_{\mathbb{R}_+^d}\mathbbm{1}_{B}(\ww_n(y,\,(i,\bbj{n-1}),\bs{n},\bbt{n}))\mathcal{T}_n(y,(i,\bbj{n-1},j),\bs{n},\bbt{n})d\bs{n}\\
&\geq\int_{W(y)}\mathbbm{1}_{B\times \mathbb{R}^{n-d}} (\mathcal{R}_y(\bs{n})) \mathcal{T}_n(y,(i,\bbj{n-1},j),\bs{n},\bbt{n})d\bs{n}\\
&=\int_{W(y)}\mathbbm{1}_{B\times \mathbb{R}^{n-d}} (\mathcal{R}_y(\bs{n})) |\operatorname{det}\partial_{\bs{n}}\mathcal{R}_y({\bs{n}})|\cdot|\operatorname{det}\partial_{\bs{n}}\mathcal{R}_y({\bs{n}})|^{-1} \mathcal{T}_n(y,(i,\bbj{n-1},j),\bs{n},\bbt{n})d\bs{n}\\
&\geq \tilde{c} \int_{W(y)}\mathbbm{1}_{B\times \mathbb{R}^{n-d}} (\mathcal{R}_y(\bs{n}))|\operatorname{det}\partial_{\bs{n}}\mathcal{R}_y({\bs{n}})|\,d\bs{n}.
\end{align*}
If we now change variables by setting $\bold{s}_n=\mathcal{R}_y(\bs{n})$ and, further, apply \eqref{e:sub}, then we can conclude that
\begin{align*}
P^n((y,i),B\times\{j\})&=\tilde{c}\int_{\mathcal{R}_y(W(y))}\mathbbm{1}_{B\times\mathbb{R}^{n-d}}(\bold{s}_n)\,d\bold{s}_n\geq \tilde{c}\int_{U_{\hat{w}}\times U_{\hat{\bold{t}}^{n-d}}}\mathbbm{1}_{B\times\mathbb{R}^{n-d}}(\bold{s}_n)\,d\bold{s}_n\\
&=\tilde{c}\,\ell_n((B\times \mathbb{R}^{n-d})\cap (U_{\hat{w}}\times U_{\hat{\bold{t}}^{n-d}}))=\tilde{c}\,\ell_{n-d}(U_{\hat{\bold{t}}^{n-d}})\ell_d(B\cap U_{\hat{w}}),
\end{align*}
Finally, we see that \eqref{e:small1} holds with $\bar{c}:=\tilde{c}\,\ell_{n-d}(U_{\hat{\bold{t}}^{n-d}})>0$.

\vspace*{0.1cm}
\textbf{Case II:} Let us now assume that $\Theta$ is an interval in $\mathbb{R}$. The proof in this case is similar to the previous one.  This time, however, we need to consider a family $\{\mathcal{R}_{y,\,\bt{n}}:\,y\in Y,\; \bt{n}\in\Theta^n \}$ of maps from $(0,\infty)^n$ into $Y\times\mathbb{R}_+^{n-d}\subset\mathbb{R}^n$, wherein $\mathcal{R}_{y,\,\bt{n}}$ is defined by
$$\mathcal{R}_{y,\,\bt{n}}(\bs{n}):=(\ww_n(y,(i,\bbj{n-1}),\bs{n},\bt{n}),\bold{t}^{n-d})\quad\text{for}\quad \bs{n}\in(0,\infty)^n.$$
Furthermore, $\mathcal{H}$ will now stand for the map  \hbox{$\mathcal{H}:(0,\infty)^n \times 
\interior Y \times \interior \Theta^n   \to (Y \times \mathbb{R}_+^{n-d}) \times Y\times\Theta^n$} (acting from an open subset of $\mathbb{R}^{2n+d}$ into itself) given by $$\mathcal{H}(\bs{n},y,\bt{n}):=(\mathcal{R}_{y,\,\bt{n}}(\bs{n}),y,\bt{n})\quad\text{for any}\quad \bs{n}\in (0,\infty)^n,\;y\in \interior Y, \; \bt{n}\in \interior \Theta^n.$$

Since the Jacobi matrix of $\mathcal{H}$ is of the form
$$\partial_{(\bs{n},y,\bt{n})}\mathcal{H}(\bs{n},y,\bt{n})=
\left[
\begin{array}{c|c}
\partial_{\bs{d}}\ww_n & \partial_{\bold{t}^{n-d}}\mathcal{W}_n \;\; \partial_{y}\mathcal{W}_n\;\;\partial_{\bt{n}}\mathcal{W}_n\\
\hline
\mathbf{0}_{d,2n} &I_{2n}
\end{array}
\right](y,(i,\bbj{n-1}),\bs{n},\bt{n}),
$$
similarly as in the previous case, we obtain
\begin{equation}
\label{e:det2a}
\operatorname{det}\partial_{(\bs{n},y,\bt{n})}\mathcal{H}(\bbs{n},\hat{y},\bbt{n})=\operatorname{det}\,\partial_{\bs{d}}\ww_n(\hat{y},\,(i,\bbj{n-1}),\bbs{n},\bbt{n}) \operatorname\neq 0.
\end{equation}
This enables us to choose an open neighbourhood \hbox{$\widehat{V}_{(\bbs{n},\hat{y},\,\bbt{n})}\subset (0,\infty)^n\times\interior Y\times \interior \Theta^n$} of the point  $(\bbs{n},\hat{y},\bbt{n})$ so that $\mathcal{H}|_{\widehat{V}_{(\bbs{n},\hat{y},\bbt{n})}}$ is a diffeomorphism from $\widehat{V}_{(\bbs{n},\hat{y},\bbt{n})}$ onto $\mathcal{H}(\widehat{V}_{(\bbs{n},\hat{y},\bbt{n})})$.

Applying \eqref{e:pos}, \eqref{e:det2a} and the continuity of  $\widehat{V}_{(\bbs{n},\hat{y},\,\bbt{n})}\ni(\bs{n},y,\bt{n})\mapsto \det \partial_{(\bs{n},y,\bt{n})}\mathcal{H}(\bs{n},y,\bt{n})$, we may find an open neighbourhood $\widetilde{V}_{(\bbs{n},\hat{y},\,\bbt{n})}\subset \widehat{V}_{(\bbs{n},\hat{y},\,\bbt{n})}$ of the point $(\bbs{n},\hat{y},\bbt{n})$ such that, for some constant $\tilde{c}>0$,
\begin{equation*}
\left|\operatorname{det}\partial_{(\bs{n},y,\,\bt{n})}\mathcal{H}({\bs{n}},y,\bt{n}) \right|^{-1}\mathcal{T}_n(y,(i,\bbj{n-1},j),\bs{n},\bt{n})\geq \tilde{c}\quad\text{for any}\quad (\bs{n},y,\bt{n})\in \widetilde{V}_{(\bbs{n},\hat{y},\,\bbt{n})},\;j\in I,
\end{equation*}
where $\mathcal{T}_n$ is defined by \eqref{def_t}. This obviously yields
\begin{equation}
\label{e:pos2aa}
\left|\operatorname{det}\partial_{\bs{n}}\mathcal{R}_{y,\bt{n}}({\bs{n}}) \right|^{-1}\mathcal{T}_n(y,(i,\bbj{n-1},j),\bs{n},\bt{n})\geq \tilde{c}\quad\text{for any}\quad (\bs{n},y,\bt{n})\in \widetilde{V}_{(\bbs{n},\hat{y},\,\bbt{n})},\;j\in I.
\end{equation}

Since $\mathcal{H}(\widetilde{V}_{(\bbs{n},\hat{y},\bbt{n})})$ is open and $((\hat{w},\hat{\bold{t}}^{n-d}),\hat{y},\bbt{n})\in \mathcal{H}(\widetilde{V}_{(\bbs{n},\hat{y},\bbt{n})})$, it follows that there exist open bounded neighbourhoods $U_{(\hat{w},\,\hat{\bold{t}}^{n-d})}\subset Y\times \mathbb{R}_+^{n-d}$, $U_{\hat{y}}\subset Y$ and $U_{\bbt{n}}\subset \Theta^n$ of the points $(\hat{w},\hat{\bold{t}}^{n-d})$, $\hat{y}$ and $\bbt{n}$, respectively, such that $U_{(\hat{w},\,\hat{\bold{t}}^{n-d})}\times U_{\hat{y}}\times U_{\bt{n}}\subset \mathcal{H}(\widetilde{V}_{(\bbs{n},\hat{y},\,\bbt{n})})$. Define $$V_{(\bbs{n},\hat{y},\,\bbt{n})}:=\mathcal{H}^{-1}(U_{(\hat{w},\,\hat{\bold{t}}^{n-d})}\times U_{\hat{y}}\times U_{\bbt{n}})$$ 
and
$$W(y,\bt{n}):=\{\bs{n}\in (0,\infty)^n:\,(\bs{n},y,\bt{n})\in V_{(\bbs{n},\hat{y},\,\bbt{n})} \}\quad\text{for any}\quad (y,\bt{n})\in U_{\hat{y}}\times U_{\bbt{n}}.$$ 
Then, arguing analogously as in Case I, we can conclude that all the sets $W(y,\bt{n})$ are open, and that $\mathcal{R}_{y,\theta_n}|_{W(y,,\bt{n})}$ is a diffeomorphism from $W(y,\bt{n})$ onto $U_{(\hat{w},\,\hat{\bold{t}}^{n-d})}$ for any $(y,\bt{n})\in U_{\hat{y}}\times U_{\bbt{n}}$. This observation, as before, enables us to choose (independently of $y$ and $\bt{n}$) open neighbourhoods $U_{\hat{w}}\subset Y$ and $U_{\hat{\bold{t}}^{n-d}}\subset\mathbb{R}_+^{n-d}$ of the points $\hat{w}$ and $\hat{\bold{t}}^{n-d}$, respectively, so that
\begin{equation}
\label{e:sub2}
U_{\hat{w}}\times U_{\hat{\bold{t}}^{n-d}} \subset U_{(\hat{w},\,\hat{\bold{t}}^{n-d})}= \mathcal{R}_{y,\theta_n}(W(y,\bt{n}))\quad\text{for any}\quad (y,\bt{n}) \in U_{\hat{y}}\times U_{\bbt{n}}.
\end{equation}

Proceeding similarly as in the first part of the proof, from \eqref{e:nstep} and \eqref{e:pos2aa} we may now deduce that, for any $B\in\mathcal{B}(Y)$, $j\in I$ and $y\in U_{\hat{y}}$,
\begin{align*}
P^n((y,i), B\times\{j\})&\geq \int_{\Theta^n}\int_{\mathbb{R}_+^d}\mathbbm{1}_{B}(\ww_n(y,\,(i,\bbj{n-1}),\bs{n},\bt{n}))\mathcal{T}_n(y,(i,\bbj{n-1},j),\bs{n},\bt{n})d\bs{n}\,\vartheta^{\otimes n}(d\bt{n})\\
&\geq\int_{U_{\bbt{n}}}\int_{W(y,\bt{n})}\mathbbm{1}_{B\times \mathbb{R}^{n-d}} (\mathcal{R}_{y,\bt{n}}(\bs{n})) \mathcal{T}_n(y,(i,\bbj{n-1},j),\bs{n},\bt{n})d\bs{n}\,\vartheta^{\otimes n}(d\bt{n})\\
&\geq \tilde{c} \int_{U_{\bbt{n}}} \int_{W(y,\bt{n})}\mathbbm{1}_{B\times \mathbb{R}^{n-d}} (\mathcal{R}_{y,\bt{n}}(\bs{n}))|\operatorname{det}\partial_{\bs{n}}\mathcal{R}_{y,\bt{n}}({\bs{n}})|\,d\bs{n}\,\vartheta^{\otimes n}(d\bt{n}).
\end{align*}
Finally, substituting $\bold{s}_n=\mathcal{R}_{y,\bt{n}}(\bs{n})$ (for every fixed $(y,\bt{n})$ separately) and applying \eqref{e:sub2}, gives
\begin{align*}
P^n((y,i),B\times\{j\})&=\tilde{c}\int_{U_{\bbt{n}}}\int_{\mathcal{R}_y(W(y,\bt{n}))}\mathbbm{1}_{B\times\mathbb{R}^{n-d}}(\bold{s}_n)\,d\bold{s}_n\,\vartheta^{\otimes n}(d\bt{n})\\
&\geq \tilde{c} \int_{U_{\bbt{n}}}\int_{U_{\hat{w}}\times U_{\hat{\bold{t}}^{n-d}}}\mathbbm{1}_{B\times\mathbb{R}^{n-d}}(\bold{s}_n)\,d\bold{s}_n\,\vartheta^{\otimes n}(d\bt{n})\\
&=\tilde{c}\,\vartheta^{\otimes n}(U_{\bbt{n}})\ell_{n-d}(U_{\hat{\bold{t}}^{n-d}})\ell_d(B\cap U_{\hat{w}}),
\end{align*}
which shows that \eqref{e:small1} holds with $\bar{c}:=\tilde{c}\,\vartheta^{\otimes n}(U_{\bbt{n}})\ell_{n-d}(U_{\hat{\bold{t}}^{n-d}})>0$ and, therefore, completes the proof.
\end{proof}

\begin{rem}
Note that, in the case where $d=1$, condition \eqref{e:rank} can be expressed in the following simple form:
$$\sum_{r=1}^n\left(\frac{\partial \ww_n}{\partial t_r}(\hat{y},(i,\bbj{n-1}),\bbs{n},\bbt{n})\right)^2>0.$$
\end{rem}

Assuming that conditions \eqref{e:pos} and \eqref{e:rank} hold with some $(\hat{y},i)\in \operatorname{int} X$, we intend to apply Proposition~\ref{prop:prop2} with $U=U_{\hat{y}}$ and $V=U_{\hat{w}}$, where $U_{\hat{y}}$ and $U_{\hat{w}}$ are the open sets guaranteed by Lemma~\ref{lem:rank}. To do this, we need to know that, for any given $P$-ergodic invariant measure $\mu_*$, the set  $U_{\hat{y}}\times\{i\}$ is uniformly accessible from an open set $\widetilde{X}\subset X$, satisfying $\mu_*(\widetilde{X})>0$, in some given number of steps, i.e. condition \eqref{e2:prop2} holds with $U=U_{\hat{y}}$ and the given~$i$ for some $m\in\n$. This is the case, for example, if the operator $P$ is asymptotically stable, and the point~$(\hat{y},i)$, verifying the desired properties, belongs to the support of the unique $P$-invariant measure.

\begin{cor}\label{cor:3}
Let $P$ and $\{P_t\}_{t\geq 0}$ stand for the Markov operator and the Markov semigroup \hbox{induced} by \eqref{e:kernel} and \eqref{e:kernel_cont}, respectively. Further, suppose that, for some  $\mu_*\in\mathcal{M}_{prob}(X)$, and for any $x\in X$, the sequence $\{P^n \delta_x\}_{n\in\n}$ is weakly convergent to $\mu_*$ (which, by Remark \ref{rem:asymp}, is equivalent to say that $P$ is asymptotically stable). Moreover, assume that all the transformations $w_{\theta}$ and $S_k(t,\cdot)$ are non-singular with respect to~$\ell_d$, and that there exists a point $(\hat{y},i)\in \operatorname{int}X \cap \operatorname{supp}\mu_*$, for which the assumptions of Lemma \ref{lem:rank} are fulfilled. Then both $\mu_*$ and $G\mu_*$, which are then unique invariant measures for $P$ and $\{P_t\}_{t\geq 0}$, respectively, are absolutely continuous with respect to $\bar{\ell}_d$.
\end{cor}
\begin{proof}
In the light of Proposition \ref {prop:prop2} and Lemma \ref{lem:rank}, it suffices to show that \eqref{e2:prop2} holds for $U=U_{\hat{y}}$ and the given $i$. Since $\hat{x}:=(\hat{y},i)\in \operatorname{supp}\mu_*$, it follows that \hbox{$\delta_{*}:=\mu_*(U\times\{i\})>0$}. Taking into account that $P^n(x,\cdot)\stackrel{w}{\to} \mu_*$ for any $x\in X$, we can apply the Portmanteau theorem \hbox{(\cite[Theorem 2.1]{b:bil})} to deduce that
$$\liminf_{n\to\infty} P^n(x,U\times\{i\})\geq \delta_{*}\quad\text{for any}\quad x\in X.$$
In particular, we therefore get $P^m(\hat{x},U\times\{i\})>\delta_*/2$ for some $m\in\mathbb{N}$. Since the operator $P$ is Feller, the map $X\ni x\mapsto P^m(x,U\times\{i\})$ is lower semicontinuous, and thus there exists an open neighbourhood of $\hat{x}$, say $\widetilde{X}$, such that $P^m(x,U\times\{i\})>\delta_*/3$ for any $x\in\widetilde{X}$. Moreover, $\mu_*(\widetilde{X})>0$, since $\hat{x}\in\supp\mu_*$. This shows that \eqref{e2:prop2} is indeed satisfied (with $\delta=\delta_*/3$) and completes the proof.
\end{proof}

The requirement $(\hat{y},i)\in \operatorname{supp}\mu_*$ is rather implicit and difficult to verify without any additional information regarding the measure $\mu_*$. Moreover, the above-stated results are limited by the assumption that the underlying operator is asymptotically stable. In the remainder of the paper, we therefore derive a somewhat more practical result, which does not require the stability, and enables one to establish the uniform accessibility of $U_{\hat{y}}\times\{i\}$ in the sense of \eqref{e2:prop2}, using a more intuitive argument, which refers directly to the component functions of the model. More precisely, given  $(\hat{y},i)\in \interior X$, we shall use the following condition:
\begin{itemize}
\phantomsection
\label{a}
\item[(A)] For any open neighbourhood $V_{\hat{y}}$ of $\hat{y}$ and any $(y,j)\in X$, there exist $n\in\n$, $\bs{n}\in\mathbb{R}_+^n$, $\bt{n}\in \Theta^n$ and, whenever $n>1$, also $\bj{n-1}\in I^{n-1}$, such that
$$
\mathcal{W}_n(y,\,(j,\bj{n-1}),\bs{n},\bt{n})\in V_{\hat{y}}\quad\text{and}\quad\pp_n(y,(j,\bj{n-1}),\bs{n},\bt{n})\Pi_n(y,(j,\bj{n-1},i),\bs{n},\bt{n})>0.$$
\end{itemize}
The following lemma, which is essentially based on \cite[Lemma 3.16]{b:benaim1}, should be treated as an intermediate result on the way to the above-mentioned  implication  $\hyperref[a]{(A)}\Rightarrow   \eqref{e2:prop2}$.
\begin{lem}\label{lem:3}
Let $\mu\in\mathcal{M}_{fin}(X)$ be an arbitrary non-zero measure. Further, suppose that condition~\hyperref[a]{(A)} holds for some $(\hat{y},i)\in \interior X$, and let $U_{\hat{y}}\subset Y$ be an arbitrary open neighbourhood of $\hat{y}$.  Then, there exist constants $\varepsilon>0,\;\beta>0,\;m\in\n$, sequences $\bbbj{m-1}\in I^{m-1}$, $\bbbs{m}\in\mathbb{R}_+^m,\;\bbbt{m}\in\Theta^m$ (the former only if $m>1$) and an open set $\widetilde{X}\subset X$ with $\mu(\widetilde{X})>0$ such that
\begin{equation}\label{e:ac}
\begin{gathered}
\mathcal{W}_m(y,(j,\bbbj{m-1}),\bs{m},\bt{m})\in U_{\hat{y}},\\
\pp_m(y,(j,\bbbj{m-1}),\bs{m},\bt{m})\Pi_m(y,(j,\bbbj{m-1},i),\bs{m},\bt{m})>\beta,
\end{gathered}
\end{equation}
whenever $(y,j)\in \widetilde{X}$, $\bs{m}\in B_{\mathbb{R}_+}(\bbbs{m},\varepsilon)$ and $\bt{m}\in B_{\Theta}(\bbbt{m},\varepsilon)$, where
$$B_{\mathbb{R}_+}(\bbbs{m},\varepsilon):=\{\bs{m}\in \mathbb{R}_+^m:\;\norma{\bs{m}-\bbbs{m}}_m<\varepsilon\}\vspace{-0.1cm},$$
$$B_{\Theta}(\bbbt{m},\varepsilon):=
\begin{cases}
\{\bt{m}\in \Theta^m:\;\|\bt{m}-\bbbt{m}\|_m<\varepsilon\} &\text{if}\;\; \Theta\;\; \text{is an interval,}\\
\{\bbbt{m}\} &\text{if}\;\;\Theta\;\;\text{is finite},
\end{cases}
$$
and $\norma{\cdot}_m$ stands for the Euclidean norm in $\mathbb{R}^m$.
\end{lem}
\begin{proof}
For any $k\in\n$, let $\mathcal{A}_k$ denote the set of all $(\bj{k-1},\bs{k},\bt{k},\beta')$, where $\bj{k-1}\in I^{k-1}$, $\bs{k}\in\mathbb{R}_+^k$, $\bt{k}\in\Theta^k$ and $\beta'>0$ (excluding the first member whenever $k=1$). Further, let $V_{\hat{y}}$ be a bounded open neighbourhood of $\hat{y}$ such that $\operatorname{cl}V_{\hat{y}}\subset U_{\hat{y}}$, and define
$$O(\bj{k-1},\bs{k},\bt{k},\beta'):=\{
(y,j)\in X:\,\mathcal{W}_k(y,(j,\bj{k-1}),\bs{k},\bt{k})\in V_{\hat{y}},\;\mathcal{Q}_k(y,(j,\bj{k-1},i),\bs{k},\bt{k})>\beta' \},$$
with
$$\mathcal{Q}_k(y,(j,\bj{k-1},i),\bs{k},\bt{k}):=\pp_k(y,(j,\bj{k-1}),\bs{k},\bt{k})\Pi_k(y,(j,\bj{k-1},i),\bs{k},\bt{k}),$$
for any $k\in\n$ and $(\bj{k-1},\bs{k},\bt{k},\beta')\in\mathcal{A}_k$.

Obviously, by continuity of the functions underlying the model, all the sets $\mathcal{O}(\cdot)$ are open. Hence, from hypothesis \hyperref[a]{(A)} it follows that 
$$\mathcal{V}:=\{O(\bj{k-1},\bs{k},\bt{k},\beta'):\,k\in\n,\; (\bj{k-1},\bs{k},\bt{k},\beta')\in\mathcal{A}_k\}.$$
is an open cover of $X$. 

Since $X$ is a Lindel\"of space (as a separable metric space) there exists a countable subcover of~$\mathcal{V}$. Consequently, we can choose sequences $\{k_r\}_{r\in\n}\subset\mathbb{N}$ and $\left\{\left(\bj{k_r-1}^{(r)},\bs{k_r}^{(r)},\bt{k_r}^{(r)},\beta_r\right)\right\}_{r\in\n}$, wherein $\left(\bj{k_r-1}^{(r)},\bs{k_r}^{(r)},\bt{k_r}^{(r)},\beta_r\right)\in\mathcal{A}_{k_r}$ for any $r\in\mathbb{N}$, so that
$$X= \bigcup_{r\in\n} O\left(\bj{k_r-1}^{(r)},\bs{k_r}^{(r)},\bt{k_r}^{(r)},\beta_r\right).$$
Now, taking into account that $\mu(X)>0$, we may find $p\in\n$ such that
$$\mu\left(O\left(\bj{k_p-1}^{(p)},\bs{k_p}^{(p)},\bt{k_p}^{(p)},\beta_p\right)\right)>0.$$
Define 
$$m:=k_p,\;\;(\bbbj{m-1},\bbbs{m},\bbbt{m},\bar{\beta}):=\left(\bj{k_p-1}^{(p)},\bs{k_p}^{(p)},\bt{k_p}^{(p)},\beta_p\right),\;\;\widetilde{X}:=O\left(\bbbj{m-1},\bbbs{m},\bbbt{m},\bar{\beta}\right).$$
Clearly, we then have
$$\mathcal{W}_m(y,(j,\bbbj{m-1}),\bbbs{m},\bbbt{m})\in V_{\hat{y}}\quad\text{and}\quad\mathcal{Q}_m(y,(j,\bbbj{m-1},i),\bbbs{m},\bbbt{m})>\bar{\beta}\quad\text{for any}\quad (y,j)\in \widetilde{X}.$$

Since $\operatorname{cl} V_{\hat{y}}\cap U_{\hat{y}}^c=\emptyset$ and $\operatorname{cl} V_{\hat{y}}$ is compact, the distance between $V_{\hat{y}}$ and $U_{\hat{y}}^c$ is positive. This,\linebreak together with continuity of $\mathcal{W}_m$ and $\mathcal{Q}_m$ with respect to $y$, $\mathbf{t}_m$ and (if $\Theta$ is an interval) $\bt{m}$, enables one to choose $\varepsilon>0$ so small that 
$$\mathcal{W}_m(y,(j,\bbbj{m-1}),\bs{m},\bt{m})\in U_{\hat{y}}\quad\text{and}\quad \mathcal{Q}_m(y,(j,\bbbj{m-1},i),\bs{m},\bt{m})>\bar{\beta}/2,$$ whenever $(y,j)\in \widetilde{X}$, $\bs{m}\in B_{\mathbb{R}_+}(\bbbs{m},\varepsilon)$ and $\bt{m}\in B_{\Theta}(\bbbt{m},\varepsilon)$. The proof is now complete.
\end{proof}

We are now in a position to establish the main result of this paper, which provides certain conditions sufficient for the absolute continuity of invariant measures for both the Markov operator~$P$ and the Markov semigroup~$\{P_t\}_{t\geq 0}$.

\begin{thm}\label{main:1}
Suppose that the transformations $w_{\theta}$,~\hbox{$\theta\in\Theta$}, and $S_k(t,\cdot)$, $k\in I$, are non-singular with respect to $\ell_d$. Further, assume that there exists a point $(\hat{y},i)\in \operatorname{int}X$ with property~\hyperref[a]{(A)}, for which
 \eqref{e:pos} and \eqref{e:rank} hold with some integer $n\geq d$ and some $(\bbj{n-1},\bbs{n},\bbt{n})\in I^{n-1}\times (0,\infty)^n\times\interior \Theta^n$ (excluding $\bbj{0}$ in the case of $n=1$). Then every ergodic invariant measure $\mu_*\in\mathcal{M}_{prob}(X)$ of the Markov operator $P$, induced by \eqref{e:kernel}, as well as the corresponding invariant measure $G\mu_*$ of the semigroup $\{P_t\}_{t\geq 0}$, generated by \eqref{e:kernel_cont}, is absolutely continuous with respect to~$\bar{\ell}_d$.
\end{thm}
\begin{proof}
Let $\mu_*\in\mathcal{M}_{prob}(X)$ be an ergodic invariant propability measure of $P$. By virtue of Lemma~\ref{lem:rank} we can choose an open neighbourhood $U_{\hat{y}}\subset Y$ of $\hat{y}$, an open set \hbox{$U_{\hat{w}}\subset Y$} and a~constant $\bar{c}>0$ so that \eqref{e1:prop2} holds with $U=U_{\hat{y}}$, $V=U_{\hat{w}}$ and the given $i$, i.e.
$$P^n(x,B\times \{j\})\geq \bar{c}\ell_d(B\cap U_{\hat{w}})\quad\text{for any}\quad x\in U_{\hat{y}}\times \{i\},\;j\in I\;\;\text{and}\;\; B\in\mathcal{B}(Y).$$

On the other hand, in view of Lemma \ref{lem:3}, we may find $\varepsilon>0$, $\beta>0$, $m\in \n$, sequences $\bbbj{m-1}\in I^{m-1}$ (if $m>1$), $\bbbs{m}\in\mathbb{R}_+^m$ , $\bbbt{m}\in\Theta^m$ and an~open set $\widetilde{X}\subset X$ with $\mu_*(\widetilde{X})>0$ such that conditions \eqref{e:ac} hold for any $z=(y,j)\in \widetilde{X}$, $\bs{m}\in B_{\mathbb{R}_+}(\bbbs{m},\varepsilon)$ and $\bt{m}\in B_{\Theta}(\bbbt{m},\varepsilon)$. Hence, appealing to \eqref{e:nstep}, we see that\vspace{-0.2cm}
$$P^{m}(z,U_{\hat{y}}\times\{i\})\geq \beta\, \vartheta^{\otimes m}(B_{\Theta}(\bbbt{m},\varepsilon)) \int_{B(\bbbs{m},\varepsilon)} \lambda^m e^{-\lambda(t_1+\ldots+t_m)}\,d\bs{m}:=\delta>0 \quad\text{for any}\quad z\in\widetilde{X},\vspace{-0.2cm}$$
which exactly means that condition \eqref{e2:prop2} holds for $U=U_{\hat{y}}$ and the given $i$. The desired absolute continuity of $\mu_*$ and $G\mu_*$ now follows from Proposition \ref{prop:prop2}.
\end{proof}

Finally, as a straightforward consequence of Theorems \ref{main:1} and \ref{cor:2}(iii), we obtain the following conclusion:

\begin{cor}\label{cor:main}
Suppose that there exists a unique invariant probability measure for the operator~$P$ or, equivalently, for the semigroup $\{P_t\}_{t\geq 0}$. Then, under the hypotheses of Theorem \ref{main:1}, both the invariant measures, that for $P$, and that for $\{P_t\}_{t\geq 0}$, are absolutely continuous with respect to $\bl$.
\end{cor}

\section{The existence and uniqueness of invariant measures}\label{sec:3} 
It is clear that to ensure the existence and uniqueness of an invariant  probability measure for the Markov operator $P$ (and therefore for the Markov semigroup $\{P_t\}_{t\geq 0}$), some additional restrictions should be imposed on the functions composing the model under consideration.

In what follows, we quote \cite[Theorem 4.1]{b:czapla_erg} (cf. also \cite[Theorem~4.1]{b:cz_kub}), which, apart from the existence of a unique $P$-invariant measure, also assures the geometric ergodicity of $P$ in the Fortet--Mourier distance on $\mathcal{M}_{prob}(X)$ (see e.g. \cite{b:las_frac} or \cite{b:dudley} for the equivalent Dudley metric). 

Assuming that $X=Y\times I$ is equipped with the metric of the form
\begin{equation}\label{e:rho_c}
\rho_c((u,i),(v,j))=\norma{u-v}+c\,\mathbf{d}(i,j)\quad\text{for any}\quad (u,i),(v,j)\in X,
\end{equation}
where $c$ is a given positive constant, the Forter--Mourier distance can be defined by
\begin{equation}
\label{def_dw}
d_{FM}(\mu,\nu):=\sup\left\{\left|\int_X f\,d(\mu-\nu)\,\right|:\;f\in \mathcal{F}_{FM}(X)\right\}\quad\text{for any}\quad\mu,\nu\in\mathcal{M}_{prob}(X),
\end{equation}
where
$$\mathcal{F}_{FM}(X):=\left\{f:X\to[0,1]:\,\, \sup_{x\neq y} \frac{|f(u)-f(v)|}{\rho_c(u,v)}\leq 1 \right\}.$$ 

It is well-known (see e.g. \cite[Theorem 8.3.2]{b:bogachev}) that the topology induced on $\mathcal{M}_{prob}(X)$ by $d_{FM}$ equals to the topology of weak convergence of probability measures (whenever $X$ is a Polish space). 

Before we formulate the above-mentioned stability result, let us emphasize that it holds with a sufficiently large constant~$c$, whose maginitude depends on the quantities occurring in the \hbox{hypotheses} to be imposed on the component functions of the model (see \cite[Section 6]{b:czapla_erg}). Let us also note that, although we have assumed that the metric on $Y$ is induced by a norm, just to stay with the framework introduced in Section \ref{sec:4} (wherein $Y$ is a closed subset of $\mathbb{R}^d$), the result remains valid for any Polish metric space (cf. \cite{b:cz_kub}).
\pagebreak
\begin{thm}[\textrm{\cite[Theorem 4.1]{b:czapla_erg}}]\label{t:erg_p}
Suppose that there exist $\alpha\in\mathbb{R}$, $L>0$ and $L_w>0$ satisfying 
\begin{equation}
\label{e:c0} LL_w+\frac{\alpha}{\lambda}<1,
\end{equation}
as well as $L_p$, $L_{\pi}$, $c_{\pi}$, $c_p>0$, a point $y^*\in Y$ and two Borel measurable functions: $\mathcal{L}:Y\to\mathbb{R}_+$, which is bounded on bounded sets, and $\varphi:\mathbb{R}_+\to\mathbb{R}_+$ satisfying
$$
\int_{\mathbb{R}_+} \varphi(t)e^{-\lambda t}\,dt<\infty,
$$
such that, for any $u,v\in Y$, the following conditions hold:
\begin{gather} 
\label{e:c1}
\|S_i(t,u) - S_i(t,v)\|\leq Le^{\alpha t}\|u-v\|\;\;\;\mbox{for any}\;\;\;i\in I,\;t\geq 0;\\
\label{e:c2}
\|S_i(t,u)- S_j(t,u)\|\leq \varphi(t)\mathcal{L}(u)\;\;\;\mbox{for any}\;\;\;i,j\in I,\;t\geq 0;\\
\label{e:c3} \sup_{y\in Y}\int_{\Theta}\int_0^{\infty}e^{-\lambda t}\|w_{\theta}(S_i(t,y^*))-y^*\|\,p_{\theta}(S_i(t,y))\,dt\,\vartheta(d\theta)<\infty\;\;\;\text{for any}\;\;\;i\in I;\\
\label{e:c4} \int_{\Theta}\|w_{\theta}(u)-w_{\theta}(v)\|\,p_{\theta}(u)\,\vartheta(d\theta)  \leq L_w\rho(u,v);\\
\label{e:c5} \int_{\Theta}|p_{\theta}(u)-p_{\theta}(v)|\,\vartheta(d\theta) \leq L_p\rho(u,v);\\ 
\label{e:c6} \sum_{k\in I} \min\{\pi_{i k}(u), \pi_{j k}(u)\}\geq c_{\pi}\;\;\text{for any}\;\; i,j\in I,\;\;\text{and}\;\; \int_{\Theta(u,v)}\min\{p_{\theta}(u),p_{\theta}(v)\}\,\vartheta(d\theta)\geq c_p,
\end{gather}
where $$\Theta(u,v):=\{\theta\in\Theta:\,\|w_{\theta}(u)-w_{\theta}(v)\| \leq L_w\rho(u,v)\}.$$ Then the Markov operator $P$ generated by \eqref{e:kernel} admits a unique invariant distribution $\mu_*$ such that \hbox{$\mu_*\in\mathcal{M}_{prob}^1(X)$}. Moreover, there exists $\beta\in(0,1)$ such that, for any $\mu\in\mathcal{M}_{prob}^1(X)$ and some constant $C(\mu)\in\mathbb{R}$, we have
\begin{equation} \label{spec_gap} d_{FM}(P^n \mu , \mu_*)\leq C(\mu)\beta^n \quad\text{for any}\quad n\in\n.
\end{equation}
In particular, $P$ is then also asymptotically stable (cf. Remark \ref{rem:asymp}).
\end{thm}

Obviously, due to Theorem \ref{thm:one-to-one}, the hypotheses of Theorem \ref{t:erg_p} also guarantee the existence and uniqueness of a probability invariant measure for the semigroup $\{P_t\}_{t\geq 0}$, generated by \eqref{e:kernel_cont}.
\begin{rem}
In paper \cite{b:czapla_erg}, the above-stated theorem is proved under the assumption that \eqref{e:c1} holds with $\varphi(t)=t$. It   is, however, easy to check that the same proof  works without any significant changes if $\varphi:\mathbb{R}_+\to\mathbb{R}_+$ is an arbitrary function such that $t\mapsto \varphi(t)\exp(-\lambda t)$ is integrable over~$\mathbb{R}_+$.
\end{rem}
\begin{rem}\label{rem:simple}
It is easy to verify (cf. \cite[Corollary 3.4]{b:czapla_erg}) that, if $\Theta$ is compact and there exists a~Borel measurable function $\psi:\mathbb{R}_+\to\mathbb{R}_+$ such that 
$$
\int_{\mathbb{R}_+} \psi(t)e^{-\lambda t}\,dt<\infty\quad\text{and}\quad\|S_j(t,y^*)-y^*\|\leq \psi(t)\quad\text{for any}\quad t\geq 0,\;j\in I,
$$
then \eqref{e:c3} holds under each of the following two conditions:
\begin{itemize}
\item[(i)] The probabilities $p_{\theta}$ are constant (i.e. they do not depend on $y\in Y$) and \eqref{e:c4} is fulfilled.
\item[(ii)] There exists $L_w>0$ such that all $w_{\theta}$, $\theta\in\Theta$, are Lipschitz continuous with the same constant~$L_w$.
\end{itemize}
\end{rem}



\section{Examples}\label{sec:5}
In this section, we shall illustrate the applicability of Theorem \ref{main:1} by analysing a simple example, inspired by \cite[Example 5.2]{b:benaim1}, wherein also the hypotheses of Theorem \ref{t:erg_p} are fulfilled. Before proceeding to this, let us, however, discuss some special cases wherein condition \hyperref[a]{(A)}, introduced prior to Lemma \ref{lem:3}, is fulfilled for some identifiable point of $X$.
\begin{rem}\label{rem:case1}
Suppose that there exist $\bar{\theta}\in\Theta$, $z\in Y$ and $i\in I$ such that the following statements are fulfilled:
\begin{itemize}
\item[(i)] $w_{\bar{\theta}}$ is a contraction satisfying $w_{\bar{\theta}}(z)=z$;
\item[(ii)] $p_{\bar{\theta}}\,(y)>0$ for any $y\in Y$;
\item[(iii)] for every $n\in\n$, there is $(j_1,\ldots,j_{n})\in I^{n}$ with $j_n=i$ such that 
\begin{equation}
\label{e:iii}
\pi_{j_{k-1}j_k}(y)>0\quad\text{for all}\quad k\in\{1,\ldots,n\} \quad\text{and}\quad y\in w_{\bar{\theta}}(Y)\quad\text{with any}\quad j_0\in I.
\end{equation}
\end{itemize}
Then condition \hyperref[a]{(A)} holds with $\hat{y}=z$ and the given $i$.
\end{rem}
\begin{proof}
Fix $(y,j)\in Y$ and $\varepsilon>0$. Letting $K<1$ denote a Lipschitz constant of $w_{\bar{\theta}}$, we can choose $n\in\n$, $n>1$, so that $K^n\norma{y-z}<\varepsilon$. According to (iii), for such an integer $n$, we may find $(j_1,\ldots,j_n)\in I^n$ with $j_n=i$ such that \eqref{e:iii} is satisfied. Taking $\bj{n-1}:=(j_1,\ldots,j_{n-1})$, $\mathbf{0}:=(0,\ldots,0)\in\mathbb{R}_+^n$ and $\bt{n}:=(\bar{\theta},\ldots,\bar{\theta})\in\Theta^n$, we now see that
$$\norma{\mathcal{W}_n(y,(j,\bj{n-1}),\mathbf{0}, \bt{n})-z}=\norma{w_{\bar{\theta}}^n(y)-w_{\bar{\theta}}^n(z)}\leq K^n\norma{y-z}<\varepsilon,$$
and $\pp_n(y,(j,\bj{n-1}),\mathbf{0},\bt{n})\Pi_n(y,(j,\bj{n-1},i),\mathbf{0},\bt{n})>0$, due to (ii) and \eqref{e:iii}.
\end{proof}

\begin{rem}\label{rem:case2}
Suppose that condition \eqref{e:c1} holds with $\alpha<0$, and that \eqref{e:c4} is satisfied. Further, assume that there exist $k\in I$, $\bar{\theta}\in\Theta$ and $i\in I$ such that the following statements are fulfilled:
\begin{itemize}
\item[(i)]  $S_k(t,z)=z$ for all $t\geq 0$;
\item[(ii)] $w_{\bar{\theta}}$ is Lipschitz continuous;
\item[(iii)]$p_{\bar{\theta}}\,(y)>0$ for any $y\in Y$;
\item[(iv)] $\pi_{jk}(y)\pi_{ki}(y)>0$ for any $j\in I$ and any $y\in w_{\bar{\theta}}(Y)$.
\end{itemize}
Then condition \hyperref[a]{(A)} holds with $\hat{y}=w_{\bar{\theta}}(z)$ and the given $i$.
\end{rem}
\begin{proof}
Let $(y,j)\in X$ and $\varepsilon>0$. Further, choose $t>0$ so that $L_{\bar{\theta}}Le^{\alpha t}\norma{w_{\bar{\theta}}(y)-z}<\varepsilon$, where $L_{\bar{\theta}}$ stands for a Lipschitz constant of $w_{\bar{\theta}}$. Now, keeping in mind that $S_j(0,u)=u$ for any $u\in Y$ and applying (ii), (i), \eqref{e:c1}, sequentially, we infer that
\begin{align*}
\norma{\mathcal{W}_2(y,(j,k),(0,t),(\bar{\theta},\bar{\theta}))-w_{\bar{\theta}}(z)}&=\norma{w_{\bar{\theta}}(S_k(t,w_{\bar{\theta}}(y)))-w_{\bar{\theta}}(z)}\leq L_{\bar{\theta}}\norma{S_k(t,w_{\bar{\theta}}(y))-z}\\
&=L_{\bar{\theta}}\norma{S_k(t,w_{\bar{\theta}}(y))-S_k(t,z)}\leq L_{\bar{\theta}}Le^{\alpha t}\norma{w_{\bar{\theta}}(y)-z}<\varepsilon.
\end{align*}
Moreover, from (iii) and (iv) it follows that 
$$\mathcal{P}_2(y,(j,k),(0,t),(\bar{\theta},\bar{\theta}))=p_{\bar{\theta}}(y)p_{\bar{\theta}}(S_k(t,w_{\bar{\theta}}(y)))>0,$$
$$\Pi_2(y,(j,,k,i),(0,t),(\bar{\theta},\bar{\theta}))=\pi_{jk}(w_{\bar{\theta}}(y))\pi_{ki}(w_{\bar{\theta}}(S_k(t,w_{\bar{\theta}}(y))))>0.$$
\end{proof}
\begin{rem}\label{rem:case2a}
Note, that in the case where $\Theta$ is finite (and $\vartheta$ is the counting measure), condition~(ii) of Remark \ref{rem:case2} can be guaranteed by assuming condition \eqref{e:c4} and a~strengthened version of (iii), namely \hbox{$p:=\inf_{y\in Y} p_{\bar{\theta}}(y)>0$}. Under these settings, $w_{\bar{\theta}}$ is Lipschitz continuous with $L_{\bar{\theta}}=p^{-1}L_w$.  
\end{rem}
\pagebreak
The latter two observations prove to be useful in analysing the announced example, given below.
\begin{ex}
Let $\alpha<0$ and $a\in\mathbb{R}\backslash\{0\}$. Consider an instance of the dynamical system introduced in Section \ref{sec:2}, with $\Theta$ satisfying the assumptions of Section \ref{sec:42}, $Y:=\mathbb{R}$, $I:=\{1,2\}$ and two semiflows, defined by
$$S_1(t,y)=e^{\alpha t}y\quad\text{and}\quad S_2(t,y)=e^{\alpha t}(y-a)+a\quad\text{for}\quad t\in\mathbb{R}_+,\;y\in\mathbb{R}.$$
Furthermore, assume that conditions  \eqref{e:c3}-\eqref{e:c5} hold for the transformations $\{w_{\theta}:\,\theta\in\Theta\}$ and the probabilities $\{p_{\theta}: \,\theta\in\Theta\}$, $\{\pi_{ij}:\, i,j\in I\}$ with $L_w=1$, as well as that
\begin{equation}
\label{e:pos3}
\inf_{y\in \mathbb{R}}\pi_{ij}(y)>0\quad\text{and}\quad \inf_{y\in \mathbb{R}}p_{\theta}(y)>0\quad\text{for any}\quad i,j\in I,\;\theta\in\Theta.
\end{equation}
Obviously, the foregoing requirement is just a strengthened form of condition \eqref{e:c6}. It is also worth noting that \eqref{e:c3} holds, for example, if $\Theta$ is compact, and $w_{\theta}$, $p_{\theta}$ satisfy at least one of conditions (i) or (ii) from Remark \ref{rem:simple}.

Clearly, the semiflows $S_1$, $S_2$ satisfy conditions \eqref{e:c1}, \eqref{e:c2} with $\alpha<0$, $L=1$, $\mathcal{L}\equiv 1$, \hbox{$\varphi(t)=|a|(1-e^{\alpha t})$}, and inequality \eqref{e:c0} is then trivially fulfilled as well. Hence, due to \hbox{Theorem \ref{t:erg_p}}, the Markov operator~$P$, corresponding to the chain given by the post-jump locations, possesses a unique invariant probability measure $\mu_*$. What is more, due to Theorem \ref{thm:one-to-one}, $\nu_*:=G\mu_*$ is the unique invariant probability measure of the transition semigroup $\{P_t\}_{t \geq 0}$, associated with the corresponding PDMP.

Suppose now that all the transformations $y\mapsto w_{\theta}(y)$, $\theta\in\Theta$, and, if $\Theta$ is an interval, also $\theta \mapsto w_{\theta}(y)$, $y\in\mathbb{R}$, are continuously differentiable and \hbox{non-singular} with respect to $\ell_1$. Furthermore, assume that, for at least one $\bar{\theta}\in\Theta$, $w_{\bar{\theta}}(a)w_{\bar{\theta}}'(w_{\bar{\theta}}(a))\neq 0$, and that the transformation $w_{\bar{\theta}}$ is Lipschitz continuous. Plainly, in the case where $\Theta$ is finite, assuming the latter is unnecessary, since the Lipschitz continuity is assured by \eqref{e:c4} and \eqref{e:pos3} (due to Remark \ref{rem:case2a}). Under the aforesaid conditions, both the invariant measures $\mu_*$ and $\nu_*$ are absolutely continuous with respect to~$\bar{\ell}_1$. To see this, first observe that $S_2(t,a)=a$ for any $t\geq 0$. Then, due to Remark~\ref{rem:case2}, condition~\hyperref[a]{(A)} holds for $(\hat{y},i):=(w_{\bar{\theta}}(a),1)$. Moreover, we have 
$$\frac{\partial}{\partial t} \mathcal{W}_1(\hat{y},1,t,\bar{\theta})=\frac{\partial}{\partial t} w_{\bar{\theta}}(S_1(t,\hat{y}))=\alpha e^{\alpha t} w_{\bar{\theta}}(a)\, w_{\bar{\theta}}'(e^{\alpha t} w_{\bar{\theta}}(a))\neq 0$$
for small enough $t>0$, which ensures that \eqref{e:rank} is satisfied with $n=d=1$, $\hat{y}=w_{\bar{\theta}}(a)$, $i=1$, $\hat{\theta}_1=\bar{\theta}$ and some sufficiently small $\hat{t}_1>0$. Obviously, \eqref{e:pos} is also fulfilled, due to \eqref{e:pos3}. Consequently, in view of Corollary \ref{cor:main}, the measures $\mu_*$ and $\nu_*$ are absolutely continuous with respect to $\bar{\ell}_1$. 
\end{ex}

It is worth noting here that the assumptions of non-singularity of the transformations $w_{\theta}$, $S_k(t,\cdot)$ and the existence of a point $(\hat{y},i)$ for which \hyperref[a]{(A)} holds are not yet sufficient for the absolute continuity of the unique $P$-invariant measure, even though the hypotheses of Theorem ~\ref{t:erg_p} are fulfilled. In other words, conditions \eqref{e:rank} and \eqref{e:pos} in Theorem \ref{main:1} cannot be omitted. This assertion can be justified by the following simple example:
\begin{ex}\label{ex:2}
Let $Y:=\mathbb{R}$, $I:=\{1\}$, $\Theta:=\{1\}$, and suppose that
$$S_1(t,y)=e^{-t}y \quad\text{and}\quad w_1(y)=y \quad\text{for any}\quad y\in Y,\; t\geq 0.$$ 
In such a case, the state space $X=\mathbb{R}\times \{1\}$ of our dynamical system can be identified with $\mathbb{R}$, and the transition law of $\{\Phi_n\}_{n\in\n_0}$, given by \eqref{e:kernel}, takes the form 
$$P(y,A)=\int_0^{\infty} \lambda e^{-\lambda t} \mathbbm{1}_A(ye^{-t})\,dt\quad\text{for any}\quad y\in \mathbb{R},\; A\in\mathcal{B}(\mathbb{R}).$$ 
Obviously, conditions \eqref{e:c0}-\eqref{e:c6} hold in this setup (\eqref{e:c3} follows directly from Remark~\ref{rem:simple}), and thus, due to Theorem \ref{t:erg_p}, there exists a unique invariant measure for $P$. Moreover, note that $S_1(t,\cdot)$, $t\geq 0$, and $w_1$ are non-singular with respect to $\ell_1$, and that condition \hyperref[a]{(A)} is fulfilled for $(\hat{y},i):=(0,1)$, since $S_1(t,0)=0$ and $w_1(0)=0$ (cf. Remark \ref{rem:case2}). On the other hand, it is easily seen that the unique $P$-invariant measure is $\delta_0$, which is singular with respect to $\ell_1$.
\end{ex}

\section*{Acknowledgements}
The work of Hanna Wojew\'odka-\'Sci\k{a}\.zko has been partly supported by the National Science Centre of Poland, grant number 2018/02/X/ST1/01518.



\bibliographystyle{plain}
\bibliography{References}
\end{document}